\documentclass[a4paper,12pt]{amsart}
\usepackage[utf8]{inputenc}

\usepackage[T1]{fontenc}
\usepackage[english]{babel}
\usepackage{color}
\usepackage{amsmath} 
\usepackage{amssymb} 
\usepackage{amsthm}
\usepackage{graphicx}
\usepackage[colorlinks=true, linkcolor=blue]{hyperref}

\usepackage{geometry}

\usepackage{enumerate} 
\usepackage{enumitem}

\usepackage{lmodern}
\usepackage{microtype}

\usepackage{subfiles}

\usepackage{soul}

\usepackage[colorinlistoftodos,prependcaption,textsize=tiny]{todonotes}

\newcommand {\R} {\mathbb{R}} 
\newcommand {\Z} {\mathbb{Z}}
\newcommand {\T} {\mathbb{T}}

\newcommand {\p} {\partial}

\newcommand{\bbN}{\mathbb{N}}

\newcommand{\bbT}{\mathbb{T}}

\newcommand{\calF}{\mathcal{F}}
\newcommand{\calG}{\mathcal{G}}

\newcommand{\calI}{\mathcal{I}}

\newcommand{\calR}{\mathcal{R}}

\newcommand{\eps}{\varepsilon}

\theoremstyle{plain}
\newtheorem{theorem}{Theorem}
\newtheorem{lemma}{Lemma}
\newtheorem{pro}{Proposition}

\newtheorem{cor}{Corollary}

\theoremstyle{definition}

\newtheorem{remark}{Remark}

\date{\today}
\begin{document}

\title[Stability of Ideal MHD Equations]{Ideal Magnetohydrodynamics Around Couette Flow: Long Time Stability and Vorticity-Current Instability}

\begin{abstract}
    This article considers the ideal 2D magnetohydrodynamic equations in a infinite periodic channel close to a combination of an affine shear flow, called Couette flow, and a constant magnetic field. This incorporates important physical effects, including mixing and coupling of velocity and magnetic field. We establish the existence and stability of the velocity and magnetic field for Gevrey-class perturbations of size $\eps$, valid up to times $t \sim \eps^{-1}$. Additionally, the vorticity and current grow as $O(t)$ and there is no inviscid damping of the velocity and magnetic field. This is similar to the above threshold case for the $3D$ Navier-Stokes \cite{bedrossian2022dynamics} where growth in `streaks' leads to time scales of $t\sim \eps^{-1}$. In particular, for the ideal MHD equations, our article suggests that for a wide range of initial data, the scenario ``induction by shear $\Rightarrow $ vorticity and current growth $\Rightarrow $ vorticity and current breakdown'' leads to instability and possible turbulences.

\end{abstract}
\author{Niklas Knobel}
\address{Imperial College London, 180 Queen's Gate, South Kensington, London SW7 2AZ}
  \email{nknobel@ic.ac.uk}


\maketitle
\setcounter{tocdepth}{1}
\tableofcontents
\section{Introduction }
In this paper, we consider the ideal 2D magnetohydrodynamic (MHD) equations
\begin{align}
  \label{MHD}
  \begin{split}
    \partial_t \tilde V + \tilde V\cdot \nabla\tilde  V+ \nabla \Pi  &= \tilde B\cdot\nabla \tilde B, \\
    \partial_t \tilde B + \tilde V\cdot\nabla \tilde B &= \tilde B\cdot\nabla\tilde  V, \\
    \nabla\cdot \tilde V=\nabla\cdot \tilde B  &= 0,\\
    (t,x,y) &\in \R_+ \times\bbT\times \R=: \Omega. 
  \end{split}
\end{align}
They model the evolution of the velocity field $\tilde V: \Omega \to \R^2 $ of a conducting fluid interacting with a magnetic field $\tilde B:\Omega \to \R^2$. The function $\Pi: \Omega \to \R$ denotes the pressure. The MHD equations are derived from the Navier-Stokes and Maxwell equations and are a common model used in plasma physics and engineering \cite{davidson_2016,freidberg2014ideal}. A fundamental question of the MHD equations is understanding the behavior in the limit of high fluid and magnetic Reynolds numbers. The ideal MHD equations correspond to the limiting case and play a crucial role in modeling fusion reactors \cite{freidberg1982ideal}. 

For the ideal MHD equations, global-in-time well-posedness is an open problem. Only local-in-time wellposedness has been established \cite{kozono1989weak,secchi1993equations, miao2006well, cobb2023elsasser}. The maximal time of existence remains an important question and there exists a blow-up criterion on vorticity and current similar to the Beale-Kato-Majda criterion \cite{caflisch1997remarks,cannone2007losing}. Under strong decay assumptions, small initial perturbations of the ideal MHD equations around a large constant magnetic field behave like the linearized system \cite{bardos1988longtime}. 

 The (ideal) MHD equations exhibit several important physical and mathematical effects and an essential challenge is to understand the long time behavior of solutions. We investigate perturbations around the combination of an affine shear flow, called Couette flow and a constant magnetic field:
\begin{align}
\begin{split}
    V_s(t,x,y)&=\begin{pmatrix}
y\\0
\end{pmatrix},\\
     B_s(t,x,y)&=\begin{pmatrix}
\alpha \\0
\end{pmatrix},\qquad \alpha \in \R. \label{eq:C&M}
\end{split}
\end{align}
The Couette flow leads to mixing by transporting perturbations and also leads to induction in the magnetic field. The constant magnetic field induces a strong coupling between velocity and magnetic field perturbations.
We are interested in perturbations around this stationary solution 
\begin{align*}
    V(t,x,y)&= \tilde V(t,x,y )- V_s, \\
    B(t,x,y)&= \tilde B(t,x,y )- B_s,
\end{align*}
which satisfy the equations
\begin{align}
  \label{MHD2}
  \begin{split}
    \partial_t V+y\p_x V + V_2e_1 &+ V\cdot \nabla V- 2\p_x \nabla \Delta^{-1} V_2 + \nabla \Pi  = B\cdot\nabla B, \\
    \partial_t B+y\p_x B- B_2e_1 &+ V\cdot\nabla B\qquad \qquad \qquad \qquad \quad \, = B\cdot\nabla V, \\
    \nabla\cdot V=\nabla\cdot B  &= 0,\\
    (t,x,y) &\in \R_+ \times\bbT\times \R=: \Omega. 
  \end{split}
\end{align}
The main result of this article is the existence, stability and growth of solutions:
\begin{theorem}\label{Thm:idealmain}
    Let $\alpha \neq 0$ and $\lambda_1 > \lambda_2 >0 $. Then for all $\tfrac 1 2 <s\le 1$ and $N\ge 5$ there exists  $c_0= c_0(\alpha,s,  \lambda_1 ,\lambda_2 ) \in (0,1)  $  such that the following holds true: \\
    Let $0<\eps<c_0 $ and consider mean and divergence-free initial data, which satisfies the following bound in Gevrey spaces
    \begin{align*}
        \Vert (V_{in} ,B_{in})  \Vert_{ \calG^{\lambda_1}}^2&:=\sum_{k\in \Z} \int \langle k,\eta\rangle^{2N } e^{2 \lambda_1\vert k,\eta \vert^s }\vert (\hat V_{in}  ,\hat B_{in})  \vert^2(k,\eta )  d\eta = \eps^2.
    \end{align*}
    Then the corresponding solution  of \eqref{MHD2} satisfies the following:
    \begin{enumerate}[label=(\Alph*)]
        \item \textbf{Existence:} The maximal time of existence $T$ satisfies  $T\ge c_0 \eps^{-1}$.
        \item \textbf{Stability:} For times $0\le t\le c_0  \eps^{-1}$ the following nonlinear stability holds:  
    \begin{align*}
        \Vert (V ,B)(t,x+yt,y) \Vert_{ \calG^{ \lambda_2} } \lesssim\eps.
    \end{align*}
\item \textbf{Growth of vorticity and current:} Consider the vorticity $W:=\nabla^\perp\cdot  V$ and current $J:= \nabla^\perp\cdot  B$. There exists a  $K>0$ such that if 
\begin{align*}
    \Vert (V_{in},B_{in})\Vert_{H^{-1} } \ge  K c_0 \eps,
\end{align*}
then
\begin{align*}
    \Vert (V,B)(t) \Vert_{L^2} &\gtrsim    \Vert (V_{in},B_{in})\Vert_{L^2}, \\
    \Vert (W,J)(t) \Vert_{L^2} &\gtrsim   \langle t \rangle \Vert (V_{in},B_{in})\Vert_{H^{-1}} 
\end{align*}
for times $0\le t\le c_0 \eps^{-1}$. 

    \end{enumerate}
\end{theorem}
The stability of the ideal MHD equations extends to the dissipative MHD equations if the dissipation parameters are small compared to the initial data: 
\begin{cor}\label{cor:dissi}
    Consider the dissipative MHD equation \eqref{MHD2dissi} with dissipation parameter $\nu,\kappa\ge 0$.  Under the same assumption as Theorem \ref{Thm:idealmain} and if the dissipation parameter satisfys additionally $\max(\nu,\kappa)  \le \eps $, then existence (A) and stability (B) hold. 
\end{cor}

\begin{remark}[On dissipation regimes]
    In Corollary \ref{cor:dissi} we extend the ideal stability result to the dissipative MHD equations if the initial data is large compared to the dissipation parameters. For fixed $\eps$ the stability holds as $\kappa, \nu\to 0 $. In other works on nonlinear stability of the MHD equations around Couette flow the initial data is chosen to be small in terms of the dissipation (see \cite{zhao2024asymptotic,Dolce,knobel2024sobolev} or the overview written later). In particular, Corollary \ref{cor:dissi} this is the first result if $\max(\nu^{\frac 12 },\kappa^{\frac 12 }) \ll  \eps $ for the full dissipative case (compared to \cite{wang2024stabilitythreshold2dmhd}) and if $\kappa \ll  \eps $ for the inviscid and resistive case (compared to \cite{zhao2024asymptotic}). 

    Dissipation is expected to significantly impact the dynamics at timescales of $t\sim \min(\nu^{-\frac 13 }, \kappa^{-\frac 1 3})$ but stability persists until $t\sim \min(\nu^{-1 }, \kappa^{-1})$. Furthermore, Corollary \ref{cor:dissi} implies that the destabilizing effect of large viscosity  $\nu\gg \kappa $ only appears on time scales $t\gg \nu^{-1}$. This destabilizing effect has been studied rigorously in \cite{knobel2024nr} and also appears in \cite{knobel2024sobolev, wang2024stabilitythreshold2dmhd}. 
    
\end{remark}

\begin{remark}[On time of stability and comparison to $3D$ Navier-Stokes equation]
    The time scale $t \sim   \epsilon^{-1}$ is expected to be optimal for stability due to nonlinear effects. In the linear dynamics, the velocity and magnetic field are (global) stable, however, the vorticity and current grow in time.  Due to the structure of the nonlinearity, nonlinear effects grow with the vorticity and current such that around the time $t \sim \eps^{-1}$ we expect that nonlinear effects dominate and secondary instabilities take over. 
    
    There are parallels to the stability of the 3D Navier-Stokes equations around Couette flow. In the above threshold case  \cite{bedrossian2022dynamics} the velocity is linearly stable in its non-average part but due growth of 'streaks' in the $x$-average, which yields increasing nonlinear effects, the stability only holds until $t\sim  \eps^{-1}$ for Gevrey perturbations of size $ \eps$. They describe the destabilizing mechanism as ``lift-up effect $\implies $ streak growth $\implies $ streak breakdown''. 
    
    For the ideal MHD equations around Couette flow, we describe the destabilizing effect as ``induction by shear $\implies $ vorticity and current growth $\implies $ vorticity and current breakdown''. By induction by shear effect, we refer to the non-transport interaction of the magnetic field perturbation and the Couette flow (i.e. $b_2 e_1 =(b\cdot\nabla) ye_1 $). Together with the constant magnetic field coupling, this leads to the growth of the vorticity and current perturbation.
    
\end{remark}

Let us comment further on the result:
\begin{itemize}    
\item 
    The nonlinear destabilizing mechanism is the ideal MHD equations' equivalent of Orr's observation \cite{orr1907stability}.  Fourier modes in nonlinearity exhibit a large interaction for specific frequencies and times, leading to growth in both modes. These resonances can lead to resonance chains which yield norm inflation in Gevrey regularity (see Subsection \ref{sec:NLgrow}).

     \item  Due to the structure of the ideal MHD equations and the growth of the vorticity and current, nonlinear effects are stronger than in related works on the Euler \cite{bedrossian2013inviscid} and inviscid Boussinesq equations \cite{bedrossian21} around Couette flow (and affine temperature profile). In particular, for the Euler and Boussinesq equations, one expects stability in Sobolev spaces for times $t\sim \eps^{-1}$.

    \item We use carefully constructed tailored unknowns that match the equation's linear and nonlinear dynamics such that the estimates reduce to their key elements. Furthermore, they allow us the choice of $\alpha \neq 0$ where all constants degenerate as $\alpha \to 0$. These tailored unknowns are constructed in two steps: First, we adapt the velocity to cancel out possible growth in the magnetic field. Second, we change to symmetric variables to use cancellations in the equation.

    \item The current and vorticity are important unknowns in the literature and in particular for wellposedness \cite{caflisch1997remarks}. The current and vorticity are expressed in terms of the tailored unknowns, making it sufficient to work with those alone.    
    
\end{itemize}

\textit{$\diamond$ Previous works on the MHD equations around Couette flow and constant magnetic field:} For the MHD equations around Couette flow and constant magnetic field, previous stability results rely on the resistivity $\kappa>0$ (or viscosity if $\nu\le \kappa$). In 2020, Liss proved the first stability result for the dissipative MHD equations around Couette flow and constant magnetic field \cite{liss2020sobolev}. He considered the fully dissipative regime, $\kappa =\nu>0$ and established the stability of the three-dimensional MHD equations for initial data small in Sobolev spaces for a large enough constant magnetic field. In two dimensions, Chen and Zi consider a shear close to Couette flow and showed stability in Sobolev spaces \cite{chen2024sobolev}. For the more general setting of $0<\kappa^3\lesssim \nu\le \kappa $, Dolce proves stability for small initial data in Sobolev spaces \cite{Dolce}. In \cite{knobel2023sobolev}, Zillinger and the author prove stability for the case of horizontal resistivity and full viscosity, $\nu = \kappa_x>0$ and $\kappa_y=0$. When resistivity is smaller than viscosity, $0<\kappa\le \nu $ small initial data are stable in Sobolev spaces, but if resistivity is very small compared to the viscosity there is growth in the magnetic field leading to norm inflation of order $\nu \kappa^{-\frac 1 3 }$ \cite{knobel2024sobolev}. This was generalized by Whang and Zang \cite{wang2024stabilitythreshold2dmhd} to all different full dissipation regimes $0<\nu, \kappa\le 1 $. In \cite{knobel2023echoes}, Zillinger and the author consider the regime of vanishing viscosity $\nu=0$ and resistivity $\kappa>0$, around traveling waves Gevrey 2 spaces are necessary and sufficient for stability.  In a corresponding nonlinear stability result,  Zhao and Zi \cite{zhao2024asymptotic} proved the almost matching nonlinear stability result for small perturbations in Gevrey $\sigma$ spaces for $1\le \sigma <2$. 
In these previous results for the MHD equations, the norm of the vorticity or current grows by $\kappa^{-\frac 1 3 }$ or $\nu\kappa^{-\frac 2 3 }$ in the non-resistive limit. Theorem \ref{Thm:ideal} holds for times $t\le c_0  \eps^{-1}$, this is the natural timescale due to nonlinear resonances and growth of the vorticity and current. All of the mentioned articles, except \cite{knobel2023echoes}, assume a large enough constant magnetic field. In upcoming work \cite{knobel2024nr}, together with Michele Dolce and Christian Zillinger, we consider the viscous and non-resistive MHD equations around Couette flow. For this setting, the viscosity reduces the interaction of velocity and magnetic field such that the induction effect leads to current growth of $O(t^2)$ instead of $O(t)$ as in the ideal case considered in this article. In particular, the viscosity destabilizes the equation compared to the ideal case, leading to shorter times of stability.

\textit{$\diamond$ Resonances and Gevrey regularity:} In plasmas or fluids, it is an important question whether perturbations of a specific solution behave as under the linear dynamics. Echoes are an experimental and numerical observed effect \cite{yu2005fluid,malmberg1968plasma,shnirelman2013long} where nonlinear effects change the dynamics at a later time. Such echoes can appear due to resonances in nonlinear terms which lead to norm inflation in Gevrey regularity. In the following, we summarize existing results on related equations using Gevrey spaces. The use of Gevrey regularity to bound resonances mechanism was first established by Villani for the Vlasov-Poisson equations \cite{Landau3}. In fluid dynamics, resonances appear for perturbations of the Couette flow. Bedrossian and Masmoudi established the stability of the Euler equations for small initial data in Gevrey spaces \cite{bedrossian2013inviscid}. Following this seminal work, many results for the Euler and Navier-Stokes equations have been established \cite{bedrossian2014enhanced,dengZ2019,dengmasmoudi2018,ionescu2020inviscid, bedrossian2020dynamics,bedrossian2022dynamics,li2022asymptotic}. For the inviscid Boussinesq equations, Gevrey perturbations of size $\eps$ are stable until times $t\sim\eps^{-2}$ \cite{bedrossian21,zillinger2023stability}. For the Boussinesq equations with viscosity but without thermal dissipation small Gevrey perturbations lead to stability \cite{masmoudi2022stability2,zillinger2021echo}. For the inviscid and resistive MHD equations, small Gevrey perturbations lead to stability \cite{knobel2023echoes,zhao2024asymptotic}.

\textit{$\diamond$ Further technical remarks:}
\begin{itemize}[noitemsep,topsep=0pt]
    
    \item  The radii of convergence $\lambda_1>\lambda_2>0$ can be chosen arbitrarly small. The $c_0$ depends on $\lambda$ such that $c_0 \lesssim \min( \lambda_1 -\lambda_2 , \lambda_2 )$.     

    \item In the proof we do not change variables according to the average velocity and magnetic field. This is not necessary, since in lower regularity we improve the estimates on the average. Therefore, the contribution of the transport by the average is small.
    
    \item As discussed in \cite{bedrossian2013inviscid} it is sufficient to consider $s \le \tfrac 3 4 $. For higher Gevrey regularity, it is sufficient to add a Gevrey multiplier of higher order and perform the estimates in lower order. 
    
\end{itemize}

\subsection*{Structure of the Article:}
The remainder of the article is structured as follows:  
\begin{itemize}
    \item In Section \ref{sec:MI} we outline the main ideas of the proof and derive the equations for the adapted unknowns.
    \item In Section \ref{sec:GS} we introduce the energy, time-dependent Fourier multiplier and the bootstrap hypothesis. Then we formulate all the necessary estimates and show that they are sufficient to prove the main theorem. Furthermore, we prove the linear estimates.
    \item In Section \ref{sec:NE} we establish the nonlinear stability estimates which are the main part of this article. To bound resonances in the nonlinearity, we split it into reaction, transport, remainder and average term according to different frequencies in Fourier space and bound them separately. Furthermore, we bound the low-frequency average by an improved rate to bound the transport by the average. 
    
    \item In Section \ref{sec:LO} we prove lower bounds on the adapted unknowns which yield the growth of the vorticity and current. 
    \item In Section \ref{sec:dissi} we show the additional estimates necessary for Corollary \ref{cor:dissi}
    
\end{itemize}

\subsection*{Notations and Conventions}
\label{sec:notation}

Let $v $ be a scalar or vector, then we define 
\begin{align*}
    \langle v \rangle &:= \sqrt{1+ \vert v \vert^2}.
\end{align*}

We write $f\lesssim g $ if $f\le C g $ for a constant $C$ independent of $\eps$ and $c_0 $ and  $f\approx g $ if $f\lesssim g$ and $g\lesssim f $.

Let $f\in L^2(\T\times \R)$, we define for $(k, \xi)\in \Z\times \R$ it's Fourier transform as
\begin{align*}
    \calF f (k,\xi) :=\hat f (k,\xi) :&= \tfrac 1{2\pi} \int_0^{2\pi} dx\int dy \ e^{-i(k x +\xi y)}f(x,y).
\end{align*}
For brevity of notation, we omit writing $\calF$.

The Lebesgue spaces are denoted as $L^p=L^p(\T\times \R)$ and the Sobolev spaces 
 as $H^N= H^N(\T\times \R )$ where $N\in \bbN$. We define the Gevrey-$\tfrac 1 s $ space with radius of convergence $r>0$ and Sobolev constant $N$ as functions $f\in L^2$ such that
\begin{align*}
     \Vert f\Vert^2_{\calG^{r}} &:=\Vert f\Vert^2_{\calG^{r;s,N  }} :=\sum_{k\in \Z} \int \langle k,\xi\rangle^{2N } e^{2r\vert k,\eta \vert^s }\vert \hat f(k,\eta )\vert^2  d\eta 
\end{align*}
is finite. Since we fix $N\in \bbN$ and  $\tfrac 1 2 < s \le 1$ we just write $\calG^{r}$. For time-dependent functions, we denote $L^p_T H^s$ as the space with the norm 
\begin{align}
    \Vert f \Vert_{L^p_TH^s}&= \left\Vert \Vert f\Vert_{H^s(\T\times \R)}\right \Vert_{L^p(0,T)}.\label{eq:lebT}
\end{align}

For $f \in L^2(\T\times\R)$, we denote the $x$-average and its $L^2$-orthogonal complement as
\begin{align*}
    f_=(y) &= \int_\bbT f(x,y) dx,\\
    f_{\neq }&= f-f_=.
\end{align*}

The time-dependent spatial derivatives are written as 
\begin{align*}
    \p_y^t&= \p_y-t\p_x,\\
    \nabla_t &= (\p_x,\p_y^t)^\top,\\
    \Delta_t &= \p_x^2 +(\p_y^t)^2,
\end{align*}
and the half Laplacians as 
\begin{align*}
    \Lambda &= (-\Delta)^{\frac 1 2 },\qquad     \Lambda_t= (-\Delta_t)^{\frac 1 2 }.
\end{align*}

The following is a list of adapted unknowns we use (See Figure \ref{fig:unknowns} for an illustration):
\begin{align*}
      p_1&=\Lambda_t^{-1}\nabla^\perp_t \cdot  v_{\neq},\\
     \tilde p_1&=\Lambda_t^{-1}\nabla^\perp_t \cdot  v_{\neq}-\tfrac 1 \alpha \p_y^t \Lambda_t^{-3}\nabla^\perp_t \cdot  b_{\neq},\\
     p_2&=\tilde p_2= \Lambda_t^{-1}\nabla^\perp_t \cdot  b_{\neq},\\
      \tilde v &= v + \tfrac 1 \alpha \p_x^{-1} b_2 e_1,\\
       \tilde b &= b.
\end{align*}
For these unknowns and a time-dependent Fourier multiplier $A$, it holds that
\begin{align}
    \Vert A\tilde p \Vert_{L^2}\approx \Vert A p \Vert_{L^2}=\Vert A (v_{\neq}, b_{\neq} )  \Vert_{L^2}\approx \Vert A (\tilde v_{\neq}, \tilde b_{\neq} )  \Vert_{L^2}.\label{eq:papprox}
\end{align}
 In particular, this holds for Sobolev weights, Gevrey weights and the weights we apply later \eqref{eq:Adef}. For  $v, b\in H^N$ and divergence-free we obtain
 \begin{align}\begin{split}
     \p_t p_1 &=\Lambda_t^{-1}\nabla^\perp_t \cdot \p_t  v_{\neq},\\
     \p_t p_2 &=\Lambda_t^{-1}\nabla^\perp_t \cdot \p_t  b_{\neq}.\label{eq:comderr}
  \end{split}\end{align}

\section{Main Ideas and Challenges of the Proof}\label{sec:MI}
In this section, we give an overview of the main mathematical effects in the ideal MHD equations. The right construction of tailored unknowns $\tilde p$ is crucial to allow for small values of $\alpha$. They are a modification of previously used unknowns \cite{Dolce, knobel2023sobolev} to allow for precise linear and nonlinear estimates. Since current and vorticity are expressed in terms of the tailored unknowns, the growth of vorticity and current is implied by the analysis of the tailored unknowns. First, in Subsection \ref{sec:cood} we change coordinates according to the Couette flow.  In Subsection \ref{sec:aunk} we discuss a more heuristic approach to deduce the tailored unknowns governed by the linearized dynamics. In Subsection \ref{sec:aunk2} we derive the nonlinear equations for the tailored unknowns. In Section \ref{sec:NLgrow} we discuss the main resonance model and provide a heuristic of why Gevrey $2$ regularity is necessary for stability.

\subsection{Change of coordinates }\label{sec:cood}
For perturbations around Couette flow it is natural the change of variables $x\mapsto x+yt$ is according to the characteristics of the Couette flow. Therefore, we introduce the unknowns 
\begin{align*}
    v(t,x,y)&= V(t,x+yt,y ), \\
    b(t,x,y)&= B(t,x+yt,y ),
\end{align*}
 For these unknowns, equation \eqref{MHD} becomes 
\begin{align}
\begin{split}
    \partial_t v + v_2 e_1 &=   \alpha \partial_x b  + b\cdot \nabla_t b- v\cdot \nabla_t v-\nabla_t \pi+  2\partial_x \Delta^{-1}_t  \nabla_t v_2 , \\
    \partial_t b - b_2 e_1 &= \alpha \partial_x v  +b\cdot \nabla_t v -v\cdot \nabla_t b,\\
    \nabla_t\cdot v&=\nabla_t\cdot b = 0,\label{eq:NLvb}
\end{split}
\end{align}
where the spatial derivatives become time-dependent and are defined as $\p_y^t=\p_y-t\p_x$, $\nabla_t = (\p_x,\p_y^t)^\intercal$ and $\Delta_t = \p_x^2+ (\p_y^t) ^2$.

\subsection{Adapted Unknowns for the Linearized Dynamics}\label{sec:aunk}
In this subsection, we derive the tailored unknowns for which we establish stability. In the following, we assume vanishing $x$-average. The linearized equation of \eqref{eq:NLvb} are
\begin{align}
\begin{split}
    \partial_t v + v_2 e_1 &=   \alpha \partial_x b +  2\partial_x \Delta^{-1}_t  \nabla_t v_2 , \\
    \partial_t b - b_2 e_1 &= \alpha \partial_x v ,\\
    \nabla_t\cdot v&=\nabla_t\cdot b = 0,\label{eq:Linvb}
\end{split}
\end{align}
where $2\partial_x \Delta^{-1}_t   v_2$ is the linear pressure. For small values of $\alpha$, due to the lack of interaction, the $-b_2 e_1$ term leads to growth. We adapt to this effect by introducing the adapted velocity 
\begin{align*}
    \tilde v &= v + \tfrac 1 \alpha \p_x^{-1} b_2 e_1 . 
\end{align*}
Then \eqref{eq:Linvb} reads
 \begin{align}\begin{split}
     \p_t \tilde v  &=  \alpha \p_x  b+  2\partial_x \Delta^{-1}_t  \nabla_t v_2,\\
     \p_t b &= \alpha \p_x \tilde  v,\\
     \nabla_t\cdot \tilde v &=\tfrac 1 \alpha b_2, \quad \nabla_t\cdot b =0. \label{eq:Lintvb}
\end{split} \end{align}
For equation \eqref{eq:Lintvb} we lose the divergence-free condition in the velocity $\tilde v$. To avoid usihe pressure we use the symmetric variables 
\begin{align*}
    \tilde p_1&= \Lambda_t^{-1} \nabla^\perp_t\cdot  \tilde v ,\\
    \tilde p_2&= \Lambda_t^{-1} \nabla^\perp_t\cdot \tilde  b.
\end{align*}
We derive the full equations in the next subsection. 

\subsection{Nonlinear Equations for the Adapted Unknowns}\label{sec:aunk2}
In this subsection, we derive the full nonlinear equations for the adapted unknowns. In contrast to the previous subsection, we derive the full nonlinear equations by the steps 
\begin{align*}
    (v,b)\to p \to \tilde p.
\end{align*}
Then, the nonlinear terms are represented to allow for precise proof.  We define the unknowns $p_1 := \Lambda_t^{-1} \nabla^\perp_t \cdot v_{\neq} $ and $p_2 := \Lambda_t^{-1} \nabla^\perp_t \cdot b_{\neq}$, thus by \eqref{eq:comderr} we obtain the equation
\begin{align}
\begin{split}
     \p_t p_1 - \p_x \p_y^t \Delta^{-1}_t p_1- \alpha \p_x p_2 &=\Lambda^{-1}_t  \nabla^\perp_t (b\nabla_t b- v\nabla_t v)_{\neq}, \\
  \p_t p_2 +\p_x \p_y^t \Delta^{-1}_t p_2 - \alpha \p_x p_1 &= \Lambda^{-1}_t \nabla^\perp_t (b\nabla_t v- v\nabla_t b)_{\neq}\label{eq:NLp}. 
\end{split}
\end{align} 
Let $\tilde p_1 = p_1- \tfrac 1{\alpha} \p_y^t \Delta^{-1}_t  p_2$, then we obtain
\begin{align*}
    \p_t p_2 &=-\p_x \p_y^t \Delta^{-1}_t p_2+ \alpha \p_x( \tilde p_1+  \tfrac 1{\alpha} \p_y^t \Delta^{-1}_t  p_2)+ \Lambda^{-1}_t \nabla^\perp_t (b\nabla_t v- v\nabla_t b)_{\neq}\\
    &= \alpha \p_x \tilde p_1 + \Lambda^{-1}_t \nabla^\perp_t (b\nabla_t v- v\nabla_t b)_{\neq}
\end{align*}
and 
\begin{align*}
    \p_t \tilde p_1 &= \p_x \p_y^t \Delta^{-1}_t (\tilde p_1 +\tfrac 1{\alpha} \p_y^t \Delta^{-1}_t  p_2)+ \alpha \p_x p_2+ \Lambda^{-1}_t  \nabla^\perp_t (b\nabla_t b- v\nabla_t v)_{\neq}\\
    &\quad - \tfrac 1 {\alpha } \p_x (\p_x^2-(\p_y^t)^2) \Delta_t^{-2} p_2 \\
    &\quad - \p_x  \p_y^t \Delta^{-1}_t\tilde p_1- \tfrac 1{\alpha} \p_y^t \Delta^{-1}_t(\Lambda^{-1}_t \nabla^\perp_t (b\nabla_t v- v\nabla_t b)) _{\neq}\\
    &= \alpha \p_x p_2 - \tfrac 1 {\alpha \p_x } \p_x^2(\p_x^2-2(\p_y^t)^2)  \Delta_t^{-2} p_2\\
    &\quad + \Lambda^{-1}_t  \nabla^\perp_t (b\nabla_t b- v\nabla_t v)_{\neq}- \tfrac 1{\alpha} \p_y^t \Delta^{-1}_t(\Lambda^{-1}_t \nabla^\perp_t (b\nabla_t v- v\nabla_t b))_{\neq}.
\end{align*}
In summary, equations \eqref{eq:NLvb} in terms of $(\tilde p_1,p_2,v_=,b_=)$ are given by
\begin{align}
\begin{split}
    \p_t \tilde p_1
    &= \alpha \p_x p_2 - \tfrac 1 {\alpha \p_x } \p_x^2(\p_x^2-2(\p_y^t)^2)  \Delta_t^{-2} p_2+ \Lambda^{-1}_t  \nabla^\perp_t (b\nabla_t b- v\nabla_t v)_{\neq}\\
    &+ \tfrac 1{\alpha} \p_y^t \Lambda^{-3}_t \nabla^\perp_t (b\nabla_t v- v\nabla_t b)_{\neq},\\
    \p_t p_2     &= \alpha \p_x \tilde p_1 + \Lambda^{-1}_t \nabla^\perp_t (b\nabla_t v- v\nabla_t b)_{\neq},\\
    \p_t v_= &= ( b\nabla_t b- v\nabla_t v)_=,\\
     \partial_t b_{=} &= (b\nabla_t v -v\nabla_t b)_= .    \label{eq:tilpI}
\end{split}
\end{align}

The following chart gives an overview of the different adapted unknowns:
\begin{figure}[hbt!]
    \caption{Relation of different unknowns. }
    \label{fig:unknowns}
\begin{center}
\begin{tikzpicture}
    \node (v) at (0,0) {$v$};
    \node (p1) at (0,-2) {$p_1$};
    \node (tv) at (5,0) {$\tilde v$};
    \node (tp1) at (5,-2) {$\tilde p_1 $};

    \draw[->] (v) -- node[above] {$v+ \tfrac 1\alpha \partial_x^{-1} e_1 b_2$} (tv);
    \draw[->] (v) -- node[right] {$\Lambda^{-1}_t\nabla_t^\perp$} (p1);
    \draw[->] (tv) -- node[left] {$\Lambda^{-1}_t\nabla_t^\perp$} (tp1);
    \draw[->] (p1) -- node[below] {$p_1 -\tfrac 1\alpha \partial_y^t \Delta_t^{-1}p_2$} (tp1);

    \node (b) at (8,0) {$b$};
    \node (p2) at (8,-2) {$p_2$};
    \node (tb) at (13,0) {$\tilde b$};
    \node (tp2) at (13,-2) {$\tilde p_2 $};

    \draw[->] (b) -- node[above] {$=$} (tb);
    \draw[->] (b) -- node[right] {$\Lambda^{-1}_t\nabla_t^\perp$} (p2);
    \draw[->] (tb) -- node[left] {$\Lambda^{-1}_t\nabla_t^\perp$} (tp2);
    \draw[->] (p2) -- node[below] {$=$} (tp2);
\end{tikzpicture}
    
\end{center}

\end{figure}

\subsection{Nonlinear Growth Mechanism}\label{sec:NLgrow}
In this section, we explain the main nonlinear effect in \eqref{eq:NLp} and \eqref{eq:tilpI}. Since the linear terms in \eqref{eq:tilpI} only lead to finite growth and $p$ and $\tilde p$ are comparable, it is sufficient to consider the nonlinearities of \eqref{eq:NLp}, in particular, we consider the model problem
\begin{align*}
     \p_t p_1  &=\Lambda^{-1}_t  \nabla^\perp_t (b\nabla_t b- v\nabla_t v)_{\neq}= -\Lambda_t^{-1}(\Lambda_t^{-1} \nabla^\perp p_2 \nabla \Lambda_t p_2 -\Lambda_t^{-1} \nabla^\perp p_1\nabla \Lambda_t p_1 )_{\neq}, \\
  \p_t p_2  &= \Lambda^{-1}_t \nabla^\perp_t (b\nabla_t v- v\nabla_t b)_{\neq}=\Lambda_t  (\nabla^\perp \Lambda_t^{-1} p_1 \nabla \Lambda_t^{-1}p_2)_{\neq} . 
\end{align*}
In the nonlinear terms, there arises a destabilizing mechanism, which is the ideal MHD equations’ equivalent of Orr’s observation \cite{orr1907stability}. For the Euler equations (c.f. \cite{bedrossian2013inviscid}), there is an interaction between high and low frequencies which leads to growth in specific modes and times which iterates to echo chains. These echo chains lead to loss of Gevrey regularity, such that Gevrey spaces are necessary for stability \cite{dengZ2019,dengmasmoudi2018}. 

For the ideal MHD equations, the growth of vorticity $w=\Lambda_t p_1$ and current $j = \Lambda_t p_2$ lead to an increase of strength of nonlinear effects. The following outlines the main growth model due to the high-low interaction on nonlinearities for the ideal MHD equations. We focus on $p_1$ (dropping the $1$ for simplicity) After a Fourier transform we obtain
\begin{align*}
    \p_t p(k,\eta)  &= -\Lambda_t^{-1} (\Lambda_t^{-1} \nabla^\perp p^{hi}\cdot \nabla \Lambda_t p^{low}) \\
    &= \sum_l \int d\xi (\eta l -k \xi)  \tfrac  {\sqrt{l^2+(\xi-lt )^2 }}{\sqrt{k^2+(\eta-kt )^2 }\sqrt{(k-l)^2+(\eta-\xi-(k-l)t )^2 }} p^{hi}(k-l,\eta-\xi) p^{low}(l,\xi).
\end{align*}
For this nonlinearity, derive the main growth model for the ideal MHD equations. If we consider $\xi \approx 0$, $l=\pm 1$ and $p(l,0)=\eps $ we will obtain on the modes $k$ and $k-1$ the model 
\begin{align*}
    \p_t p(t,k,\eta) 
    &\approx\eps \langle t \rangle  \tfrac \eta {k(k-1)} \tfrac 1{\sqrt{1+(\frac \eta k -t )^2 }\sqrt{(1+(\frac \eta{k-1} -t )^2 }} p(k-1,\eta),\\
    \p_t p(t,k-1,\eta) 
    &\approx\eps\langle t \rangle   \tfrac \eta {k(k-1) } \tfrac 1{\sqrt{1+(\frac \eta k -t )^2 }\sqrt{(1+(\frac \eta{k-1} -t )^2 }} p(k,\eta).
\end{align*}
Close to resonant times $t\in \tfrac \eta k +\tfrac 1 2 [-\tfrac \eta{k(k+1)},\tfrac \eta{k(k-1)}]=:I_k $ we obtain that 
\begin{align*}
    \tfrac \eta {k(k-1)} \tfrac 1 {\sqrt{1+(\frac \eta{k-1} -t )^2 }}\approx 1 .
\end{align*}
Therefore, for fixed $\eta$ and $\eps \langle t\rangle \approx c_0 $ the interaction model corresponds to 
\begin{align}
\begin{split}
    \p_t p(t,k) 
    &=c_0  \tfrac 1{\sqrt{1+(\frac \eta k -t )^2 }} p(k-1,\eta),\\
    \p_t p(t,k-1) 
    &=c_0    \tfrac 1{\sqrt{1+(\frac \eta k -t )^2 }} p(k,\eta).    \label{eq:mgp}
\end{split}
\end{align}
For $q^\pm(t)=p(t,k)\pm p(t,k-1)$ the main growth model looks like 
\begin{align}
    \p_t q^\pm &= \pm c_0  \tfrac 1{\sqrt{1+(\frac \eta{k} -t )^2 }}q^\pm.\label{eq:mgq}
\end{align}
For initial data $p(t_k,k-1)=0$ we obtain 
\begin{align*}
    q^\pm (t_{k-1}) &= \exp \left( \pm c_0 \int_{I_k}\tfrac 1{\sqrt{1+(\frac \eta{k} -t )^2 }}d\tau\right)\\
    &\approx \exp(\pm c_0 \sinh^{-1} ( \tfrac \eta {k^2} ))p(t_k, k).
\end{align*}
Therefore, for the mode $p(t_{k-1},k-1)$ we obtain that
\begin{align*}
    p(t_{k-1},k-1) &= \tfrac 1 2 (q^+(t_{k-1})-q^-(t_{k-1})) \\
    &= \sinh(c_0 \sinh^{-1} ( \tfrac \eta {k^2} ))p(t_k, k).
\end{align*}
Then we use that $\sinh(c_0\sinh^{-1} ( \tfrac \eta {k^2} ))\ge \tfrac {c_0} 2  ( \tfrac \eta {k^2} )^{c_0} $ to infer 
\begin{align*}
    p(t_{k-1},k-1) &\ge \tfrac {c_0} 2  ( \tfrac \eta {k^2} )^{c_0} p(t_k, k).
\end{align*}
Therefore, $p(t_k, k)$ yields growth of $p(t_{k-1},k-1)$ which we can iterate $k \to k-1$. For $\eta \gg c_0^{-1}$ and a mode $k_0 \approx \sqrt {c_0 \eta}\gg 1  $ we obtain for the resonance chain $k_0\to k_0-1 \to \dots \to 1 $ by Sterlings approximation the asymptotics
\begin{align*}
    \prod_{1\le k \le \sqrt \eta}  (\tfrac {c_0 \eta }{k^2} )^{c_0 }= (\tfrac {( c_0 \eta)^k}{(k!)^2})^{c_0}  \sim \eta^{-\frac {c_0} 2 } \exp ( C c_0 \sqrt {c_0 \eta} ). 
\end{align*}

\begin{remark}[Choice of unknowns]\label{rem:NLtu}
    For this heuristic, we choose the unknown $p_1$. The main reason for this is that it reduces the complexity and makes it more comparable to growth models in the literature (in particular \cite{bedrossian2013inviscid,bedrossian21}). In the proof, we switch to the tailored unknowns $\tilde p$ and then estimate the nonlinearity in terms of $v$ and $b$. A low and high-frequency interaction model of the nonlinearities in \eqref{eq:tilpI} yields the same Gevrey $2$ growth. However, the calculation is more complicated, needs additional steps, and is therefore less suited for a heuristic presentation. Conversely, for the nonlinear estimates, it is more useful to work with the nonlinearity in \eqref{eq:tilpI} to avoid unnecessary technical steps. 
\end{remark}

\section{Stability of the Nonlinear Equations}\label{sec:GS}
In this section, we establish the main Theorem \ref{Thm:idealmain}. We switch to the unknowns $v(t,x,y)= V(t,x+yt,y )$ and $b(t,x,y)= B(t,x+yt,y )$ and for these unknowns Theorem \ref{Thm:idealmain} states: 

\begin{theorem}\label{Thm:ideal}
    Let $\alpha \neq 0$ and $\lambda_1 > \lambda_2 >0 $. Then for all $\tfrac 1 2 <s\le 1$ and $N\ge 5$ there exists  $c_0= c_0(\alpha,s, \lambda_1 ,\lambda_2 ) \in (0,1)  $  such that the following holds true: \\
    Let $0<\eps<c_0 $ and consider mean and divergence-free initial data, which satisfies the following Gevrey bound
    \begin{align*}
        \Vert (v_{in} ,b_{in})  \Vert_{ \calG^{\lambda_1}}^2&:=\sum_{k\in \Z} \int \langle k,\eta\rangle^{2N } e^{2 \lambda_1\vert k,\eta \vert^s }\vert (\hat v_{in}  ,\hat b_{in})  \vert^2(k,\eta )  d\eta = \eps^2.
    \end{align*}
    Then the corresponding solution  of \eqref{eq:NLvb} satisfies the following:
    \begin{enumerate}[label=(\Alph*)]
        \item \textbf{Existence:} The maximal time of existence $T$ satisfies  $T\ge c_0 \eps^{-1}$.
        \item \textbf{Stability:} For times $0\le t\le c_0  \eps^{-1}$ the following nonlinear stability holds:  
    \begin{align*}
        \Vert (v ,b)(t) \Vert_{ \calG^{ \lambda_2} } \lesssim\eps.
    \end{align*}
\item \textbf{Growth of vorticity and current:} Consider the vorticity $w:=\nabla^\perp_t v$ and current $j:= \nabla^\perp_t b$. There exists a  $K>0$ such that if 
\begin{align*}
    \Vert (v_{in},b_{in})\Vert_{H^{-1} } \ge  K c_0 \eps,
\end{align*}
then
\begin{align*}
    \Vert (v,b)(t) \Vert_{L^2} &\gtrsim    \Vert (v_{in},b_{in})\Vert_{L^2}, \\
    \Vert (w,j)(t) \Vert_{L^2} &\gtrsim   \langle t \rangle \Vert (v_{in},b_{in})\Vert_{H^{-1}} 
\end{align*}
for times $0\le t\le c_0 \eps^{-1}$. 

    \end{enumerate}
\end{theorem}

To prove Theorem \ref{Thm:ideal} we employ a bootstrap method. That is, we establish energy estimates with time-dependent Fourier weights, which is a by now classical technique \cite{bedrossian2013inviscid,bedrossian2016sobolev,liss2020sobolev,bedrossian21}. We use the energies
\begin{align}\begin{split}
    E=\Vert A   (\tilde p, v_=, b_=)\Vert_{L^2}^2,\\
     E_0=\Vert A^{lo}   (v_=,b_=)\Vert_{L^2}^2,\label{Energy}
\end{split}\end{align}
with the adapted unknowns 
\begin{align}\begin{split}
    \tilde p_1&=\Lambda_t^{-1}\nabla^\perp_t \cdot  v_{\neq}-\tfrac 1 \alpha \p_y^t \Lambda_t^{-3}\nabla^\perp_t \cdot  b_{\neq},\\
    \tilde p_2&= \Lambda_t^{-1}\nabla^\perp_t \cdot  b_{\neq},\label{eq:tilp}
\end{split}\end{align}
which satisfies the equation 
\begin{align*}
    \p_t \tilde p_1
    &= \alpha \p_x \tilde p_2 - \tfrac 1 {\alpha \p_x } \p_x^2(\p_x^2-2(\p_y^t)^2)  \Delta_t^{-2} \tilde  p_2+ \Lambda^{-1}_t  \nabla^\perp_t (b\nabla_t b- v\nabla_t v)_{\neq}\\
    &\quad + \tfrac 1{\alpha} \p_y^t \Lambda^{-3}_t \nabla^\perp_t (b\nabla_t v- v\nabla_t b)_{\neq},\\
    \p_t \tilde  p_2     &= \alpha \p_x \tilde p_1 + \Lambda^{-1}_t \nabla^\perp_t (b\nabla_t v- v\nabla_t b)_{\neq},\\
    \p_t v_= &= ( b\nabla_t b- v\nabla_t v)_=,\\
     \partial_t b_{=} &= (b\nabla_t v -v\nabla_t b)_=.   
\end{align*}
We refer to Subsection \ref{sec:aunk} and \ref{sec:aunk2} for a heuristic discussion and derivation of the equation. To control the energies, the time-dependent Fourier weights are constructed in the form
\begin{align}\begin{split}
    A(t,k,\eta)&=  (mJ) (t,k,\eta )\langle k,\eta  \rangle^N  \exp( \lambda(t) \vert k,\eta  \vert^{s } ), \\
    \tilde A(t,k,\eta)&=(m\tilde   J) (t,k,\eta )\langle k,\eta  \rangle^N  \exp( \lambda(t) \vert k,\eta  \vert^{s } ), \\
    A^{lo}(t,\eta)&=  J (t,0,\eta )\langle \eta  \rangle^{N-1} \exp( \lambda(t) \vert\eta  \vert^{s } ).\label{eq:Adef}
\end{split}\end{align}
Here weight $J$ is used to bound nonlinear resonances and is an adaption of the weight in \cite{bedrossian2013inviscid} to our dynamics. Let  $\rho>0$, we define $J$ by 
\begin{align*}
    J(t,k,\eta)&= \tfrac {e^{8 \rho \vert \nu\vert^{\frac 1 2 }}} {q(t,\eta ) } +  e^{8 \rho \vert k\vert^{\frac 1 2 }},\\
    \tilde J(t,k,\eta)&= \tfrac {e^{8 \rho \vert \nu\vert^{\frac 1 2 }}} {q(t,\eta ) },
\end{align*}
and $q(\eta)$ is defined in Appendix \ref{sec:q}. The function $\lambda(t)$ is a decreasing function and corresponds to a loss in the radius of convergence. It is determined through the constraint 
\begin{align*}
    \lambda(0)&= \lambda_0, \qquad \qquad\lambda_0\ge\rho \left(250+\tfrac 2 {s-\frac 1 2}\right),  \\
    \p_t\lambda (t) &=-  \rho \tfrac {1} {\langle t\rangle^{\frac 34 +\frac s2  }}.
\end{align*}
The multiplier $m$ is used to bound linear terms and is an adaption of classical weights to our system. The frequency cut is crucial to avoid additional commutator estimates in the nonlinear term. The $m$ defined as
\begin{align*}
    m(t,k,\eta ) &= 
    \begin{cases}
    \exp \left(- \tfrac 1 {\alpha \vert k\vert } \int_{0 }^t \tfrac 1 {1+(\frac \eta k -\tau )^2  }d\tau \right)     & k\neq 0 \text{ and } \sqrt {\vert \eta\vert} \le 10 c_0 \eps^{-1}, \\
    1 &\text{else.}
    \end{cases}
\end{align*}

Let $C^\ast>0$, then for a time $t>0$ the bootstrap hypothesis states 
\begin{align}
\begin{split}
    E(t)  &+ \Vert \vert \dot \lambda\vert\Lambda^{\frac s2} A (\tilde p, v_=,b_=)\Vert_{L^2_t L^2}^2+   \Vert \sqrt{ \tfrac {\p_t q} q} \tilde A  (\tilde p, v_=,b_=)  \Vert_{L^2_t L^2}^2\le C^\ast \eps^2,\\
    E_0(t)&+  \Vert \vert \dot \lambda\vert \Lambda^{\frac s2} A^{lo}( v_= , b_= ) \Vert_{L^2_t L^2}^2+   \Vert \sqrt{ \tfrac {\p_t q} q} A^{lo}( v_= , b_= )  \Vert_{L^2_t L^2}^2 \le C^\ast (\ln(e+t ))^2c_0^{-2} \eps^4.\label{eq:Boot}
\end{split}
\end{align}
In the following, we prove that for $C^\ast$ large enough the estimate \eqref{eq:Boot} holds for $0\le t\le c_0 \eps^{-1}$. To do this, we assume \eqref{eq:Boot} holds till a maximal time $t^\ast<c_0 \eps^{-1}$ then we improve the $\le $ to $<$. By local wellposedness, we obtain a contradiction to the maximality of $t^\ast$. 

From \eqref{eq:papprox} and \eqref{eq:Boot} it follows directly, that
\begin{align}
\begin{split}
    \Vert A (v,b)  \Vert_{L^\infty_t L^2}  &+ \Vert \vert \dot \lambda\vert\Lambda^{\frac s2} A (v,b)\Vert_{L^2_t L^2}+  \Vert \sqrt{ \tfrac {\p_t q} q} \tilde A  (v,b) \Vert_{L^2_t L^2}\lesssim\eps\label{eq:Bootvb}.
\end{split}
\end{align}
In the following lemma, we compute the time derivative and introduce the linear and nonlinear terms to be estimated in the following. 
\begin{lemma}\label{eq:energyderi}
    Under the assumptions of Theorem \ref{Thm:ideal} and supposing that \eqref{eq:Boot} holds for for some time $T>0$, it holds for $t\in(0,T)$, that 

\begin{align}\begin{split}
     \p_t  E(t)  &+2 \vert \dot \lambda\vert \Vert \tilde A\Lambda^{\frac s2} (\tilde p,v_=,b_=)\Vert_{L^2}^2+  2 \Vert \sqrt{\tfrac {\p_t q} q} A(\tilde p,v_=,b_=) \Vert_{L^2}^2\\
    &=2L+ 2NL+2 ONL   \label{eq:energyhi}  
\end{split}\end{align}
and 
\begin{align}\begin{split}
    \p_t E_0(t)   &+2 \vert \dot \lambda\vert\Vert \Lambda^{\frac s2} A ( v_=,b_=)\Vert_{ L^2}^2+  2 \Vert \sqrt{\tfrac {\p_t q} q} A(v_=,b_=) \Vert_{L^2}^2\\
    &=2NL^{lo} .\label{eq:energylo}
\end{split}\end{align}
Here, we introduce the following notation: 
\begin{itemize}
    \item The linear terms are
    \begin{align}\begin{split}
        L&=\vert \langle A\tilde p_1  , \tfrac 1 {\alpha \p_x } \p_x^2(\p_x^2-2(\p_y^t)^2)  \Delta_t^{-2} A \tilde p_2\rangle\vert - \Vert \sqrt{-\tfrac {\p_t m}m } \tilde p \Vert_{L^2}^2.\label{L}
    \end{split}\end{align}
    \item The main nonlinearities are 
    \begin{align}\begin{split}
        NL&=\langle Av, A(b\nabla_t b- v\nabla_t v)-b\nabla_t Ab+ v\nabla_t Av\rangle \\ 
    &\quad + \langle A  b, (b\nabla_t v- v\nabla_t b)-b\nabla_t A v+ v\nabla_t A b)\rangle .\label{NL}\end{split}
    \end{align}
    \item The other nonlinearities are
    \begin{align}\begin{split}
        ONL&=\langle A\tfrac 1{\alpha} \p_y^t \Lambda^{-2}_t  b  , A (b\nabla_t b- v\nabla_t v)_{\neq}\rangle \\
    &\quad + \langle A\tilde p_1  , \tfrac 1{\alpha} \p_y^t \Lambda^{-2}_t(\Lambda^{-1}_t \nabla^\perp_t (b\nabla_t v- v\nabla_t b))_{\neq}  \rangle  .\label{ONL}
    \end{split}\end{align}
    \item The low frequent nonlinearities are
    \begin{align}\begin{split}
        NL^{lo} &=\langle Av_{=} , A(b\nabla_t b- v\nabla_t v)_=\rangle \\
    &\quad + \langle A  b_{=}, (b\nabla_t v- v\nabla_t b)_=\rangle.\label{NLlo}
    \end{split}\end{align}
\end{itemize}
    
\end{lemma}
\begin{proof}
    The identities \eqref{eq:energyhi} and \eqref{eq:energylo} follow by direct computation. For the non average part, we derive the energy $\Vert A (v_=,b_=)  \Vert_{L^2}^2$, integrate by parts in $\Lambda_t^{-1} \nabla^\perp_t$ and use $\tilde p_1 = p_1- \tfrac 1{\alpha} \p_y^t \Delta^{-1}_t  p_2$,  $v_{\neq}  =-\nabla^\perp_t \Lambda_t^{-1} p_1  $ and $b_{\neq}  =-\nabla^\perp_t \Lambda_t^{-1} p_2$ to infer
\begin{align*}
    \p_t  \Vert A\tilde p  \Vert_{L^2}^2 &+2 (-\dot   \lambda )  \Vert A\Lambda^{\frac s2} \tilde p \Vert_{L^2}^2+  2 \Vert  \sqrt{ \tfrac {\p_t q} q} \tilde A \tilde p  \Vert_{L^2}^2\\
    &=2\langle A\tilde p_1  , \tfrac 1 {\alpha \p_x } \p_x^4 \Delta_t^{-2} A p_2\rangle - 2\Vert \sqrt{-\tfrac {\p_t m}m } \tilde p \Vert_{L^2}^2  \\
    &\quad +2\langle A\tilde p_1  , A\Lambda^{-1}_t  \nabla^\perp_t (b\nabla_t b- v\nabla_t v)_{\neq}- \tfrac 1{\alpha} \p_y^t \Delta^{-1}_t(\Lambda^{-1}_t \nabla^\perp_t (b\nabla_t v- v\nabla_t b))_{\neq}  \rangle \\
    &\quad +2 \langle A  p_2 , \Lambda^{-1}_t \nabla^\perp_t (b\nabla_t v- v\nabla_t b)_{\neq}\rangle\\
    &=2\langle A\tilde p_1  , \tfrac 1 {\alpha \p_x } \p_x^4 \Delta_t^{-2} A p_2\rangle - \Vert \sqrt{-\tfrac {\p_t m}m } \tilde p \Vert_{L^2}^2  \\
    &\quad +2\langle Av_{\neq} , A(b\nabla_t b- v\nabla_t v)-b\nabla_t Ab+ v\nabla_t Av\rangle \\
    &\quad +2 \langle A  b_{\neq}, (b\nabla_t v- v\nabla_t b)-b\nabla_t A v+ v\nabla_t A b)\rangle \\
    &\quad +2\langle A\tfrac 1{\alpha} \p_y^t \Lambda^{-2}_t  b  , A (b\nabla_t b- v\nabla_t v)_{\neq}\rangle \\
    &\quad + 2\langle A\tilde p_1  , \tfrac 1{\alpha} \p_y^t \Lambda^{-2}_t(\Lambda^{-1}_t \nabla^\perp_t (b\nabla_t v- v\nabla_t b))_{\neq}  \rangle.     
\end{align*}
For the average, we compute the time derivative as 
\begin{align*}
    \p_t \Vert A (v_=,b_=)  \Vert_{L^2}^2  &+2 (-\dot \lambda)\Vert \Lambda^{\frac s2} A ( v_=,b_=)\Vert_{ L^2}^2+  2 \Vert \sqrt{\tfrac {\p_t q} q} A(v_=,b_=) \Vert_{L^2}^2\\
    &=2\langle Av_{=} , A(b\nabla_t b- v\nabla_t v)\rangle \\
    &\quad +2 \langle A  b_{=}, (b\nabla_t v- v\nabla_t b)\rangle.
\end{align*}
Combining these calculations yields \eqref{eq:energyhi}. Equation \eqref{eq:energylo} follows similarly by with the weight $A^{lo}$ and the $x$-average. 
\end{proof}
To apply the bootstrap, we bound the linear and nonlinear terms in the Lemma \ref{eq:energyderi}. The linear bounds are:
\begin{pro}[Linear estimates]\label{pro:lin}
      Consider the assumptions of Theorem \ref{Thm:ideal} and let $c_0,\rho>0$, $0<\eps<c_0 $ and assume that for $ 0\le t^\ast\le c_0 \eps^{-1}$ the bootstrap assumption \eqref{eq:Boot} holds for all $0\le t \le t^\ast$. Then the following estimates hold
    \begin{align}
    \int_0^t \vert L \vert \ d \tau   &\le \tfrac 1 {9} (C\eps)^2.\label{eq:Linest}
\end{align}
\end{pro}
The nonlinear bounds are:
\begin{pro}[Nonliner estimates]\label{pro:NLest}
     Consider the assumptions of Theorem \ref{Thm:ideal} and let $c_0,\rho>0$, $0<\eps<c_0 $ and assume that for $ 0\le t^\ast\le c_0 \eps^{-1}$ the bootstrap assumption \eqref{eq:Boot} holds for all $0\le t \le t^\ast$. Then there exists a $C=C(\alpha )$ such that the following estimates hold
    \begin{align}
       \int_0^t \vert  NL \vert  \ d\tau&\le Cc_0 (1+\tfrac 1\rho )  \eps^2 \label{eq:Boot1},\\
        \int_0^t \vert ONL \vert \ d\tau&\le C c_0 \eps^2,\label{eq:Boot2}
    \end{align}
    and 
    \begin{align}
        \int_0^t \vert NL^{lo}\vert \  d\tau &\le C  \ln^2(e+t )c_0^{-1}(1+\tfrac 1\rho ) \eps^4 \label{eq:Boot4}.
    \end{align}
\end{pro}

The next proposition establishes that the $L^2$ norm of the $\tilde p$ is approximately constant under suitable conditions. 
\begin{pro}[Lower estimates]\label{pro:lest}
 Assume that $\rho \gtrsim c_0 $ and  \eqref{eq:Boot} holds for all $0\le t\le  c_0 \eps^{-1} $, then there exists a $\tilde K>0$ such that if the initial data satisfies $\Vert \tilde p_{in} \Vert_{L^2} \ge \tilde  K \eps c_0 $ then for all  $0\le t\le  c_0 \eps^{-1} $ it holds 
\begin{align}
    \Vert \tilde p \Vert_{L^2} (t)\approx \Vert \tilde p_{in} \Vert_{L^2}  . \label{eq:idellow2}
\end{align}
\end{pro}

In analogue to the local well-posedness theory for the Euler equation in Gevrey spaces the ideal MHD equations are locally well-posed in Gevrey regularity, see discussion of \cite{bedrossian2013inviscid} for more details. With Propositions \ref{pro:lin}, \ref{pro:NLest} and \ref{pro:lest}, we prove Theorem \ref{Thm:ideal}:
\begin{proof}[Proof of Theorem \ref{Thm:ideal}]
    We prove existence (A) and stability (B) by the same steps. First, we need to choose $c_0, \rho$ and $\lambda_0$ to apply the bootstrap. Let $\rho$ be small and $\lambda_0 $ be large enough, such that $16 \rho + \tfrac 2 {s-\frac 1 2 } \rho \le \lambda_1 -\lambda_2 $ and $\lambda_0\ge \lambda_2$. Then it holds that
     \begin{align}
         \Vert \cdot \Vert_{G^{\lambda_2}}\le \Vert A \cdot \Vert_{L^2}\le\Vert \cdot \Vert_{G^{\lambda_1}}.\label{eq:comparison}
     \end{align}
     Then we choose $c_0$ small enough, such that $c_0  (1+\tfrac 1\rho )\le \tfrac 1 {10}  C^{-1}$, where $C$ is the constant in Proposition \ref{pro:NLest}. 
     
     For  $0<\eps <c_0$, by local existence, there exists a time $T >0$ such that \eqref{eq:Boot} holds. We assume for the sake of contradiction that there exists a maximal time $t^\ast<c_0 \eps^{-1}$ such that such that \eqref{eq:Boot} holds. Due to the choice of $c_0$ we obtain with Proposition \ref{pro:lin} and \ref{pro:NLest} that \eqref{eq:Boot} holds with a strict inequality. By local wellposedness, this contradicts the maximality of $t^\ast$ and therefore we infer existence (A) and stability (B). The growth of the vorticity and current (C) is a consequence of \eqref{eq:papprox}, Proposition \ref{pro:lest} and 
    \begin{align*}
        \Vert w,j\Vert_{L^2}\ge  \Vert \Lambda_t \Lambda p\Vert_{L^2}\gtrsim   t  \Vert \tilde p\Vert_{H^{-1}}\gtrsim t  \Vert  p\Vert_{H^{-1}}.
    \end{align*}
\end{proof}
The remainder of this article is dedicated to the proofs of Propositions \ref{pro:lin}, \ref{pro:NLest} and \ref{pro:lest}.

\subsection{Linear Terms} In this subsection, we prove Proposition \ref{pro:lin}. By Plancherel's Theorem, we obtain 
\begin{align*}
    \tfrac 1\alpha \langle  A \tilde p_1  , \p_x(\p_x^2-2(\p_y^t)^2) \Delta_t^{-2} A  \tilde  p_2 \rangle    &\le  3\sum_{k\neq 0} \int d\eta \ \tfrac 1 { \alpha \vert k\vert  } \tfrac 1 {1+(t-\frac \eta k)^2} \vert A \tilde p_1\vert  (k,\eta )\vert A  \tilde  p_2\vert(k,\eta ) 
\end{align*}
If $\sqrt{\vert  \eta\vert} \le  10\eps^{-1} $ we obtain by the definition of $m$, that 
\begin{align*}
     \tfrac 1 {\vert\alpha k\vert} \tfrac 1 {(1+(t-\frac \eta k)^2)^2}= -\tfrac {\p_t m}m .
\end{align*}
If $\sqrt {\vert  \eta\vert} >  {10}c_0 \eps^{-1} \ge 10 t  $ we obtain that either $k\ge \sqrt {\vert \eta \vert} $ or $\vert \tfrac \eta k -t \vert \ge  \tfrac 9{10}\sqrt {\vert \eta \vert}$ and therefore 
\begin{align*}
    \vert k\vert (1+(t-\tfrac \xi k )^2) \ge \sqrt {\vert  \eta\vert} >  9 c_0 \eps^{-1}.
\end{align*}
Thus,
\begin{align*}
    \vert\tfrac 1\alpha \langle  A \tilde p_1  , \p_x^3\Delta_t^{-2} A \tilde  p_2 \rangle\vert    &\le \Vert \sqrt{-\tfrac {\p_t m }m } A  \tilde p \Vert_{L^2}^2 + \tfrac 1{9 c_0}\eps \Vert A \tilde p\Vert_{L^2}^2.
\end{align*}
Integrating in time and using \eqref{eq:Boot} yields
\begin{align*}
    \int_0^t \vert\tfrac 1\alpha \langle  A \tilde p_1  , \p_x^3\Delta_t^{-2} A  \tilde  p_2 \rangle\vert  d \tau   &\le \Vert \sqrt{-\tfrac {\p_t m }m } A  \tilde  p \Vert_{L^2L^2}^2 + \tfrac 1 {9} (C\eps)^2,
\end{align*}
which concludes Proposition \ref{pro:lin}. 

\section{Nonlinear Estimates }\label{sec:NE}
In this section, we prove the nonlinear estimates of Proposition \ref{pro:NLest} by establishing the estimates \eqref{eq:Boot1}, \eqref{eq:Boot2} and \eqref{eq:Boot4}. 
It is often not necessary to distinguish between $v$ and $b$. Therefore we write $a^1,a^2,a^3\in \{v,b\} $ and $\tilde a^1,\tilde a^2,\tilde a^3\in \{\tilde v,b\} $ if the specific choice of $v$, $\tilde v$ and $b$ doesnt matter. The major part of this section is dedicated to the estimate \eqref{eq:Boot1}, which is the most important estimate and requires further splitting. We separate \eqref{NL} into 
\begin{align}\begin{split}
     \vert NL\vert &\le \vert \langle Aa^1 ,A (a^2\nabla_t a^3)-a^2\nabla_t  A a^3\rangle\vert \\
 &= \vert\langle Aa^1  ,A (a^2_{\neq}\nabla_t a^3)-a^2_{\neq}\nabla_t  A a^3\rangle\vert\\
 &+\vert\langle Aa^1 ,A (a^2_=\nabla_t a^3)-a^2_= \nabla_t  A a^3\rangle\vert\\
 &= NL_{\neq}+NL_{=}.\label{eq:NL=neq}
\end{split}\end{align}
In the following, we fix $i=i[a^2]$, such that 
\begin{align*}
    p_i &= \Lambda^{-1}_t \nabla^\perp_t \cdot  a^2. 
\end{align*}
We split the term without average into reaction, transport and remainder term 
\begin{align}\begin{split}
    NL_{\neq}
 &\le \sum_{k-l\neq 0 } \iint d(\xi ,\eta )  \tfrac {\vert \eta l- k \xi\vert \vert A(k,\eta )- A(l,\xi) \vert }{\sqrt{(k-l)^2 + (\eta - \xi -(k-l)t )^2 }} \vert Aa^1\vert (k,\eta ) \vert p_i\vert (k-l,\eta-\xi) \vert a^3\vert (l,\xi) \\
 &\qquad \qquad\qquad \cdot \left(\textbf{1}_{\Omega_R}+\textbf{1}_{\Omega_T}+\textbf{1}_{\Omega_\calR}\right) \\
&= R + T +\calR \label{eq:NLneq}
\end{split}\end{align}
according to the sets 
\begin{align}\begin{split}
    \Omega_R &=\{ (k,\eta,l ,\xi ): \ \vert k-l,\eta -\xi \vert\ge 8 \vert l, \xi\vert\}\cap  \tilde \Gamma , \\
    \Omega_T &=\{ (k,\eta,l ,\xi ): \ 8\vert k-l,\eta -\xi \vert\le  \vert l, \xi\vert \}  ,\\
    \Omega_\calR  &=\{ (k,\eta,l ,\xi ): \ \tfrac 1 8  \vert l, \xi\vert\le \vert k-l,\eta -\xi \vert\le  8 \vert l, \xi\vert  \}\\
    &\cup \left(\tilde \Gamma^c\cap \{ (k,\eta,l ,\xi ): \ \vert k-l,\eta -\xi \vert\ge 8 \vert l, \xi\vert\}\right), \label{eq:Omegas}
\end{split}\end{align}
where
\begin{align*}
    \tilde \Gamma &= \{4\langle k \rangle  \le \vert \eta \vert \text{ and } \ 4\vert k-l \vert \le \vert \eta-\xi \vert \}. 
\end{align*}
The reason for $\tilde \Gamma$ is, that we can exchange $J$ and $\tilde J$ on $\tilde \Gamma$. The set $\tilde \Gamma^c$ is part of the remainder term to avoid unnecessary splittings in the reaction term.

The nonlinear growth mechanism of Subsection \ref{sec:NLgrow} is included in the reaction term $R$. Due to the choice of the tailored unknowns $\tilde p$ we directly apply the Fourier weight $q$. For the transport term $T$, we use improved estimates of $\vert A(k,\eta )- A(l,\xi) \vert$, here the frequency cut of $m$ is crucial. Furthermore, the improved estimates of the low-frequency average in Subsection \ref{sec:lf}, must match the estimates done in Subsection \ref{sec:NL=}.

 The remainder of this section is dedicated to estimating the nonlinear terms. The $R$ is estimated in Subsection \ref{sec:R}, $T$ is estimated in Subsection \ref{sec:T}, $\calR$ is estimated in Subsection \ref{sec:calR} and $NL_=$ is estimated in Subsection \ref{sec:NL=} which yield estimate \eqref{eq:Boot1}. Then the estimates of \eqref{eq:Boot2} are done in Subsection \ref{sec:ONL}. The low frequent estimate \eqref{eq:Boot4} is done in Subsection \ref{sec:lf}.

\subsection{Reaction Term}\label{sec:R} In this subsection we estimate the reaction term in \eqref{eq:NLneq}. For the set $\Omega_R$ we obtain with Lemma \ref{lem:Jest} and \ref{lem:useest}, that 
\begin{align}\begin{split}
    \vert A(k,\eta )- A(l,\xi)\vert &\lesssim A(k-l ,\eta -\xi )\tfrac {J(k,\eta)+J(l,\xi)}{J(k-l,\eta-\xi )} e^ { \lambda(t)(\vert k,\eta \vert^{s }-\vert k-l,\eta-\xi \vert^{s })}\\
    &\lesssim A(k-l ,\eta -\xi ) e^{\frac 1 2  \lambda(t)\vert l,\xi \vert^{s }}.\label{eq:RAdiff}
\end{split}
\end{align}
We consider the sets 
\begin{align*}
\Gamma_1 &=\{(\eta,\xi):  4\lfloor \sqrt {\vert \eta-\xi\vert}\rfloor   \le t \le \tfrac 3 2  \vert\eta-\xi\vert \text{ and }  \vert \eta -\xi \vert \ge 8\vert  \xi\vert \} ,\\
\Gamma_2 &=\{ (\eta,\xi):   t \le 4 \lfloor\sqrt {\vert \eta-\xi\vert} \rfloor\text{ and }  \vert \eta -\xi \vert \ge 8 \vert  \xi\vert \} ,\\
\Gamma_3 &=\{ (\eta,\xi):   \tfrac 3 2  \eta\le t \text{ and }  \vert \eta -\xi \vert \ge 8 \vert  \xi\vert \} ,\\
\Gamma_4 &=\{(\eta,\xi): \vert \eta -\xi \vert \le 8 \vert  \xi\vert \}, 
\end{align*}
and by using \eqref{eq:RAdiff} we split the reaction term according to these sets into
\begin{align*}
    R\lesssim&  \sum_{k-l\neq 0 } \iint d(\xi ,\eta ) \textbf{1}_{\Omega_R}    \tfrac {\vert \eta l- k \xi\vert  }{\sqrt{(k-l)^2 + (\eta - \xi -(k-l)t )^2 }}  \vert Aa^1_j\vert (k,\eta ) \vert A p_i\vert (k-l,\eta-\xi) e^{\frac 1 2  \lambda(t)\vert l,\xi \vert^{s }}\vert  a^3_j\vert (l,\xi) \\
    &\qquad \qquad \cdot(\textbf{1}_{\Gamma_1 } +\textbf{1}_{\Gamma_2 } + \textbf{1}_{\Gamma_3 } + \textbf{1}_{\Gamma_4 } )\\
    &=R_1 +R_2+ R_3 +R_4.
\end{align*}
The $\Gamma_1$ set is the region where the resonances appear and we use the adapted Gevrey weights, while the $\Gamma_2$, $\Gamma_3$ and $\Gamma_4$ cut out regions, where we bound the reaction therm in different ways. In the following, we estimate all terms separately.\\
\textbf{Bound on $R_1$:}  We use $\eta l- k \xi= (\eta - \xi-(k-l)t )l +(k-l)t l- k \xi  $ to split 
\begin{align*}
    R_1\le&\sum_{k-l\neq 0 } \iint d(\xi ,\eta ) \textbf{1}_{\Omega_R}  \textbf{1}_{\Gamma_1 } \left( \tfrac { \vert (\eta - \xi-(k-l)t )l \vert }{\sqrt{(k-l)^2 + (\eta - \xi -(k-l)t )^2 }} + \tfrac {\langle t\rangle\vert k-l\vert\vert l,\xi\vert^2  }{\sqrt{(k-l)^2 + (\eta - \xi -(k-l)t )^2 }} \right)\\
    &\qquad \qquad \cdot \vert Aa^1\vert (k,\eta ) \vert A p_i\vert (k-l,\eta-\xi) e^{\frac 1 2  \lambda(t)\vert l,\xi \vert^{s }}\vert  a^3\vert (l,\xi) \\\
    =& R_{1,1 }+R_{1,2}.
\end{align*}
We estimate directly  $R_{1,1}$ by 
\begin{align}
     \vert R_{1,1}  \vert &\lesssim \Vert A a^1\Vert_{L^2}  \Vert A p_i\Vert_{L^2} \Vert A a^3\Vert_{L^2}.\label{eq:R11}
\end{align}
To estimate $R_{1,2}$, we use that on the set $\Gamma_1\cap \Omega_R$ we obtain that 
\begin{align*}
    2\max (\lfloor \sqrt{ \vert\eta-\xi\vert }\rfloor , \lfloor \sqrt{\vert\eta\vert }\rfloor)\le t \le 2 \min(\vert \eta\vert ,\vert \eta-\xi\vert )
\end{align*}
 and $\vert \eta\vert, \vert \eta-\xi\vert >1$. Therefore, for fixed $t$, $\eta$ and $\xi$  there exist $n$ and $m$ such that $t\in I_{n,\eta-\xi}\cap I_{m,\eta}$. With Lemma \ref{lem:qabl} we obtain
\begin{align*}
    \tfrac {\p_t q }q(\eta-\xi  )  &\approx \tfrac \rho{1+\vert t-\frac {\eta-\xi } {n } \vert  }.
\end{align*} 
Furthermore, for $t\in I_{n,\eta-\xi}$ we infer
\begin{align*}
    \tfrac 1{1+\vert t-\frac {\eta-\xi} {k-l } \vert  }\le \sup_{\tilde n\neq0 } \tfrac 1{1+\vert t-\frac {\eta-\xi } {\tilde n } \vert  }=\tfrac 1{1+\vert t-\frac {\eta-\xi } {n } \vert  }&=  \tfrac {\p_t q }q(\eta-\xi  ).
\end{align*}
Since $\tfrac 1 2\vert \eta-\xi\vert \le  \eta \le 2\vert \eta-\xi\vert  $ we obtain with Lemma \ref{lem:qcha}, that
\begin{align*}
    \sqrt{\tfrac {\p_t q }q}(\eta-\xi )&\lesssim \left(\sqrt{\tfrac {\p_t q }q}(\eta )+ \tfrac {\eta^{\frac s 2 }}{\langle t \rangle^{s} }\right)\langle \xi \rangle.
\end{align*}
This yields the estimate 
\begin{align}
    \textbf{1}_{\Gamma_1\cap \Omega_R}\tfrac 1{1+\vert t-\frac {\eta-\xi} {k-l } \vert  }&\le \textbf{1}_{\Gamma_1\cap \Omega_R} \sqrt{\tfrac {\p_t q }q}(\eta-\xi )\left(\sqrt{\tfrac {\p_t q }q}(\eta )+ \tfrac {\eta^{\frac s 2 }}{\langle t \rangle^{s} }\right)\langle \xi \rangle.\label{eq:resest}
\end{align}

With Lemma \ref{lem:Jest} we obtain on $\tilde \Gamma $ for the weight $J$, that
\begin{align}\begin{split}
    J(k,\eta )&\le 2 \tilde J(k,\eta ),\\
    J(k-l,\eta-\xi )&\le 2 \tilde J(k-l,\eta-\xi ).    \label{eq:R1J}
\end{split}
\end{align}
Therefore, with \eqref{eq:resest} and \eqref{eq:R1J} we obtain the estimate 
\begin{align}
     \vert R_{1, 2}  \vert \lesssim \tfrac  t \rho \left( \Vert \tilde  A\sqrt{\tfrac {\p_t q }q} a^1\Vert_{L^2}+ \tfrac 1{\langle t \rangle^{s} } \Vert A\Lambda^{\frac s2}a^1\Vert_{L^2}\right)  \Vert \tilde A \sqrt{\tfrac {\p_t q }q}p_i\Vert_{L^2} \Vert A a^3\Vert_{L^2}. \label{eq:R12}
\end{align}
\textbf{Bound on $R_2$:} For $t\le 4\sqrt{\vert \eta-\xi\vert  } $, we obtain either $\vert k-l \vert  \ge \tfrac 1 {10} \sqrt{\vert \eta-\xi\vert} $ or $\vert \eta- \xi-(k-l)t \vert \ge \tfrac 1 {10} \sqrt{\vert \eta-\xi\vert}$ and so 
\begin{align*}
     \textbf{1}_{\Gamma_2\cap \Omega_R}\tfrac1{{\sqrt{(k-l)^2 + (\eta - \xi -(k-l)t )^2 }}}&\lesssim \textbf{1}_{\Gamma_2\cap \Omega_R}\tfrac 1 {\sqrt{\vert  \eta-\xi, k-l \vert   }}.
\end{align*}
Therefore,  we obtain 
\begin{align*}
    \textbf{1}_{\Gamma_2\cap \Omega_R}\tfrac {\vert \eta l-k\xi\vert}{{\sqrt{(k-l)^2 + (\eta - \xi -(k-l)t )^2 }}}&\lesssim \textbf{1}_{\Gamma_2\cap \Omega_R} \vert \eta ,k \vert^{\frac s 2 } \vert \eta-\xi ,k-l \vert^{\frac s 2 }  \tfrac {\vert \eta -\xi , k -l\vert^{1-s}}{\sqrt{\vert  \eta-\xi, k-l \vert} }\vert l,\xi \vert\\
    &\le \textbf{1}_{\Gamma_2\cap \Omega_R}\tfrac {  \vert \eta ,k \vert^{\frac s 2 } \vert \eta-\xi ,k-l \vert^{\frac s 2 } }{\langle t \rangle^{ \frac 1 2(s-\frac 12) }}\vert l,\xi \vert.
\end{align*}
Thus we obtain by the last inequality, that 
\begin{align}
     \vert R_2 \vert &\le   t \tfrac 1{t^{\frac 34 +\frac s 2  }}\Vert A\Lambda^{\frac s2}a_1 \Vert_{L^2}\Vert A\Lambda^{\frac s2}p_i \Vert_{L^2}\Vert Aa_3 \Vert_{L^2}.\label{eq:R2}
\end{align}
\textbf{Bound on  $R_3$:} On $\Gamma_3$ with $t\ge \tfrac 32  (\eta-\xi) $ we obtain
\begin{align*}
    \textbf{1}_{\Gamma_3\cap \Omega_R}\tfrac {\vert \eta l-k\xi\vert }{\sqrt{(k-l)^2 + (\eta - \xi -(k-l)t )^2 }}&\lesssim\textbf{1}_{\Gamma_3\cap \Omega_R}\tfrac {\vert \eta-\xi\vert}{\langle t \rangle }\vert l,\xi\vert  \lesssim \textbf{1}_{\Gamma_3\cap \Omega_R}\tfrac { \vert \eta-\xi \vert^{\frac s 2 }\vert \eta \vert^{\frac s 2 }}{\langle t \rangle^{s }}\langle l,\xi\rangle ^2
\end{align*}
and so we deduce 
\begin{align}
     \vert R_3 \vert \lesssim \langle t\rangle \tfrac 1{\langle t \rangle^{1+s }}\Vert A\Lambda^{\frac s2} a^1 \Vert_{L^2 }\Vert A\Lambda^{\frac s2}p_i \Vert_{L^2 }\Vert A a^3 \Vert_{L^2 }. \label{eq:R3}
\end{align}
\textbf{Bound on $R_4$:} From  $\vert \eta -\xi \vert\le 8 \vert  \xi\vert$ we infer on $\Omega_R$, that 
\begin{align*}
    \textbf{1}_{\Gamma_4\cap \Omega_R}\tfrac {\vert \eta l-k\xi\vert}{\sqrt{(k-l)^2 + (\eta - \xi -(k-l)t )^2 }}&\le \textbf{1}_{\Gamma_4\cap \Omega_R}8 \tfrac {\langle l,\xi\rangle^2 }{\sqrt{1 + (\tfrac {\eta - \xi }{k-l}-t )^2 }}\lesssim \textbf{1}_{\Gamma_4\cap \Omega_R}  \langle l,\xi\rangle^2  . 
\end{align*}
Therefore, we infer 
\begin{align}
     \vert R_4 \vert  &\lesssim  \Vert Aa^1 \Vert_{L^2}  \Vert Ap_i \Vert_{L^2}  \Vert Aa^3 \Vert_{L^2} .\label{eq:R4}
\end{align}
\textbf{Conclusion for $R$ term:} Combining estimates \eqref{eq:R11}, \eqref{eq:R12}, \eqref{eq:R2}, \eqref{eq:R3} and \eqref{eq:R4} 
we obtain
\begin{align*}
    \vert R\vert&\lesssim \Vert A a^1\Vert_{L^2}  \Vert A p_i\Vert_{L^2} \Vert A a^3\Vert_{L^2}\\
    &\quad + t \left( \Vert \tilde  A\sqrt{\tfrac {\p_t q }q} a^1\Vert_{L^2}+ \tfrac 1{\langle t \rangle^s } \Vert A\Lambda^{\frac s2} a^1\Vert_{L^2}\right)  \Vert \tilde A \sqrt{\tfrac {\p_t q }q}p_i\Vert_{L^2} \Vert A a^3\Vert_{L^2}\\
    &\quad + t \tfrac 1{t^{\frac 34 +\frac s 2  }}\Vert A\Lambda^{\frac s2}a_1 \Vert_{L^2}\Vert A\Lambda^{\frac s2}p_i \Vert_{L^2}\Vert Aa_3 \Vert_{L^2}.
\end{align*}
By integrating in time and using \eqref{eq:Bootvb} we infer 
\begin{align*}
    \int_0^t \vert R \vert d\tau &\lesssim t (1+\tfrac 1 \rho)\eps^3\le (c_0+\tfrac {c_0} \rho )  \eps^2.
\end{align*}

\subsection{Transport Term }\label{sec:T} In this subsection estimate the transport term in \eqref{eq:NLneq}, 
\begin{align*}
   T&= \sum_{k-l\neq 0  } \iint d(\xi ,\eta ) \textbf{1}_{\Omega_T} \tfrac {\vert \eta l- k \xi\vert \vert A(k,\eta )- A(l,\xi)\vert  }{\sqrt{(k-l)^2 + (\eta - \xi -(k-l)t )^2 }} \vert Aa^1 \vert (k,\eta ) \vert p_i\vert (k-l,\eta-\xi) \vert a^3\vert (l,\xi).
\end{align*}
We recall that 
\begin{align*}
    A(k,\eta)&=J(k,\eta) \langle k,\eta  \rangle^N \exp( \lambda(t) \vert k,\eta  \vert^{s } ) 
\end{align*}
and write the difference 
\begin{align*}
    A(k,\eta )- A(l,\xi)&= A(l,\xi ) (e^{\lambda(t) (\vert k,\eta  \vert^{s }- \vert l,\xi   \vert^{s })}-1 )\\
    &\quad +A(l,\xi ) e^{\lambda(t) (\vert k,\eta  \vert^{s }- \vert l,\xi   \vert^{s })}(\tfrac {J(k,\eta )}{J(l,\xi) }-1 ) \tfrac {\langle k,\eta  \rangle^N}{\langle l,\xi  \rangle^N} \tfrac {m(k,\eta)}{m(l,\xi)}\\
    &\quad +A(l,\xi ) e^{\lambda(t) (\vert k,\eta  \vert^{s }- \vert l,\xi   \vert^{s })}( \tfrac {\langle k,\eta  \rangle^N}{\langle l,\xi  \rangle^N}-1)\tfrac {m(k,\eta)}{m(l,\xi)}\\
    &\quad +A(l,\xi ) e^{\lambda(t) (\vert k,\eta  \vert^{s }- \vert l,\xi   \vert^{s })}( \tfrac {m(k,\eta)}{m(l,\xi)}-1). 
\end{align*}
Furthermore, on $\Omega_T$ with Lemma \ref{lem:useest} we obtain the estimate 
\begin{align}
    e^{\lambda(t) (\vert k,\eta  \vert^{s }- \vert l,\xi   \vert^{s })}\lesssim e^{\frac12\lambda(t) \vert k-l, \eta-\xi\vert^s }.\label{eq:TGev}
\end{align}
Also it holds that 
\begin{align}
     \tfrac {\vert k-l\vert }{\sqrt{(k-l)^2 + (\eta - \xi -(k-l)t )^2 }}\tfrac {\vert k-l\vert  }{\sqrt{(k-l)^2 + (\eta - \xi )^2 }}\lesssim \langle t \rangle^{-1} \label{eq:invdamp}.
\end{align}
Therefore, using this splitting and estimates we split the transport term into 
\begin{align*}
   T&\lesssim \langle t \rangle^{-1}  \sum_{k-l\neq 0  } \iint d(\xi ,\eta ) \textbf{1}_{\Omega_T}  \big(\vert e^{\lambda(t) (\vert k,\eta  \vert^{s }- \vert l,\xi   \vert^{s })}-1\vert e^{-\frac12\lambda(t) \vert k-l, \eta-\xi\vert^s }+\vert \tfrac {J(k,\eta )}{J(l,\xi) }-1\vert\\
   &\qquad\qquad \qquad\qquad \qquad\qquad \qquad +\vert  \tfrac {\langle k,\eta  \rangle^N}{\langle l,\xi  \rangle^N}-1\vert+\vert \tfrac {m(k,\eta)}{m(l,\xi)}-1\vert  \big)\\
  &\qquad \qquad \qquad \qquad \cdot \vert \eta l- k \xi\vert e^{\frac12\lambda(t) \vert k-l, \eta-\xi\vert^s } \vert Aa^1 \vert (k,\eta ) \vert p_i\vert (k-l,\eta-\xi) \vert A a^3\vert (l,\xi)  \\
   &=T_1+T_2 +T_3+T_4. 
\end{align*}
\textbf{Bound on  $T_1$:} We use $\vert e^x -1\vert \le \vert x\vert  e^{\vert x\vert} $ and Lemma \ref{lem:useest} to infer 
\begin{align*}
    \vert e^{\lambda(t) (\vert k,\eta  \vert^{s }- \vert l,\xi   \vert^{s })}-1 \vert&\lesssim  (\vert k,\eta  \vert^{s }- \vert l,\xi   \vert^{s })e^{\lambda(t) (\vert k,\eta  \vert^{s }- \vert l,\xi   \vert^{s })}\\
    &\lesssim \tfrac {\vert k -l,\eta-\xi \vert } {\vert k ,\eta \vert^{1-s } }e^{ \frac 1 2 \lambda(t)\vert k-l,\eta -\xi  \vert^{s }}.
\end{align*}
Therefore, we deduce 
\begin{align}
    \vert T_1\vert  &\lesssim \langle t\rangle^{-1}\Vert A \Lambda^{\frac s2}  a^1 \Vert_{L^2}  \Vert Ap_i \Vert_{L^2}  \Vert A \Lambda^{\frac s2} a^3 \Vert_{L^2}.\label{eq:T1}
\end{align}
\textbf{Bound on  $T_2$:} For $t \le \min \left(\sqrt {\vert \eta\vert}, \sqrt {\vert  \xi \vert}\right)  $ we use Lemma \ref{lem:Jcomm} to infer 
\begin{align*}
    \vert \tfrac {J(k,\eta ) }{J(l,\xi ) }-1 \vert\lesssim \tfrac { \vert k-l , \eta-\xi \vert }{\vert \eta , k \vert^{\frac 12 } } e^{100\rho\vert k-l,\eta-\xi\vert^{\frac 12 } }.
\end{align*}
For $4\eta \le \vert l\vert$, Lemma \ref{lem:Jcomm2} yields
\begin{align*}
    \vert \tfrac {J(k,\eta ) }{J(l,\xi ) }-1\vert 
    &\lesssim \tfrac {\vert k-l\vert}{\vert k \vert^{\frac 12 } }e^{8\rho  \vert k-l\vert^{\frac 1 2 }}\lesssim \tfrac { \vert k-l , \eta-\xi \vert^2 }{\vert \eta , k \vert^{\frac 1 2 } } e^{8\rho \vert k-l,\eta-\xi\vert^{\frac 12 } }.
\end{align*}
For $t \ge \min \left(\sqrt {\vert \eta\vert}, \sqrt {\vert  \xi \vert}\right)$ and $4\vert \eta\vert\ge   \vert l \vert$, we obtain 
\begin{align*}
    \tfrac {\vert \eta l - k \xi\vert  }{\sqrt {(k-l)^2 +( \eta- \xi - (k-l ) t )^ 2}} &\lesssim \vert k, \eta\vert^{\frac s 2 } \vert l, \xi \vert^{\frac s 2 } \tfrac {\vert \eta\vert ^{1-s}}{\langle t \rangle}\vert k-l,\eta-\xi \vert^2 \\
    &\lesssim \vert k, \eta\vert^{\frac s 2 } \vert l, \xi \vert^{\frac s 2 }\tfrac {t}{\langle t \rangle^{2s} }\vert k-l,\eta-\xi \vert^2 .
\end{align*}
Combining this three estimates and \eqref{eq:invdamp}, we infer
\begin{align*}
    \tfrac {\vert \eta l - k \xi \vert}{\sqrt {(k-l)^2 +( \eta- \xi - (k-l ) t )^ 2}}  \vert \tfrac {J(k,\eta ) }{J(l,\xi ) }-1\vert  &\lesssim  \tfrac {t}{\langle t \rangle^{2s} }\vert l,\xi\vert^{\frac s 2 }\vert k,\eta \vert^{\frac s 2 } \vert k-l , \eta-\xi \vert^2 e^{100 \rho\vert k-l,\eta-\xi\vert^{\frac 12 } }
\end{align*}
and thus with \eqref{eq:TGev} and \eqref{eq:invdamp} we deduce the estimate 
\begin{align}
    \vert T_2\vert  &\lesssim  \tfrac {t}{\langle t \rangle^{2s} }\Vert A\Lambda^{\frac s2} a^1 \Vert_{L^2}  \Vert   Ap_i \Vert_{L^2}  \Vert A\Lambda^{\frac s2} a^3 \Vert_{L^2}. \label{eq:T2}
\end{align}
\textbf{Bound on  $T_3$:} On $\Omega_T$ we obtain $\vert k,\eta\vert \approx \vert l,\xi\vert$ and so
\begin{align*}
    \left\vert \tfrac {\langle k,\eta  \rangle^N}{\langle l,\xi  \rangle^N}-1\right\vert\lesssim \tfrac {\langle k-l ,\eta-\xi \rangle }{\langle l ,\xi \rangle }.
\end{align*}
Using this, we infer 
\begin{align}
    \vert T_3 \vert &\lesssim \langle t \rangle^{-1}\Vert Aa^1 \Vert_{L^2}  \Vert Ap_i \Vert_{L^2}  \Vert Aa^3 \Vert_{L^2}. \label{eq:T3}
\end{align}
\textbf{Bound on  $T_4$:} With Lemma \ref{lem:mest} we obtain, that 
\begin{align*}
     \vert \tfrac {m(k,\eta)}{m(l,\xi)}-1\vert &\le \tfrac {\vert k-l\vert} {\min(\vert k \vert, \vert l \vert) }\textbf{1}_{\min(\sqrt \eta, \sqrt \xi)\le 10 c_0 \eps^{-1} }.
\end{align*}
Therefore, we estimate 
\begin{align*}
     \vert \eta l-k\xi  \vert \vert \tfrac {m(k,\eta)}{m(l,\xi)}-1\vert &\le \vert k,\eta \vert^{\frac s 2 }\vert l,\xi \vert^{\frac s 2 } c_0 \eps^{-1} t^{-(2s-1)}\vert k-l,\eta-\xi \vert\vert k-l \vert 
\end{align*}
and with  \eqref{eq:invdamp} we deduce 
\begin{align*}
    \tfrac {\vert \eta l- k \xi\vert  }{\sqrt{(k-l)^2 + (\eta - \xi -(k-l)t )^2 }}&\le c_0 \eps^{-1} \langle t\rangle^{-2s} \vert k,\eta \vert^{\frac s 2 }\vert l,\xi \vert^{\frac s 2 }\vert k-l,\eta-\xi \vert^2.
\end{align*}
Thus we obtain
\begin{align}
   \vert T_4\vert &\lesssim c_0 \eps^{-1} \langle t \rangle^{-2s} \Vert A\Lambda^{\frac s2} a^1\Vert_{L^2}\Vert A p_i\Vert_{L^2}\Vert A\Lambda^{\frac s2}  a^3\Vert_{L^2}. \label{eq:T4}
\end{align}

\textbf{Conclusion for $T$ term:} Combining estimates \eqref{eq:T1}, \eqref{eq:T2}, \eqref{eq:T3} and \eqref{eq:T4} we obtain that 
\begin{align*}
    \vert T \vert &\lesssim \tfrac {t}{\langle t \rangle^{2s} }\Vert A \Lambda^{\frac s2} a^1 \Vert_{L^2}  \Vert   Ap_i\Vert_{L^2}  \Vert A \Lambda^{\frac s2} a^3 \Vert_{L^2}\\
    &+ \langle t \rangle^{-1}\Vert Aa^1 \Vert_{L^2}  \Vert Ap_i \Vert_{L^2}  \Vert Aa^3 \Vert_{L^2}\\
    &+ c_0 \eps^{-1} \langle t \rangle^{-2s} \Vert A\Lambda^{\frac s2} a^1\Vert_{L^2}\Vert A p_i\Vert_{L^2}\Vert A\Lambda^{\frac s2}  a^3\Vert_{L^2}. 
\end{align*}
Using \eqref{eq:Bootvb} we infer by integrating in time
\begin{align*}
    \int_0^t \vert T\vert d\tau &\lesssim (c_0+\tfrac {c_0} \rho)   \eps^2.
\end{align*}
\subsection{Remainder Term}\label{sec:calR} On $\Omega_\calR$ with Lemma \ref{lem:useest} and \ref{lem:Jcomm} and  $\lambda\ge 250 \rho $ we obtain 
\begin{align*}
    \vert \eta l-k \xi \vert \vert A(k,\eta )-A(l,\xi )\vert \lesssim  \  {\langle \tfrac {\eta-\xi} {k-l}\rangle^{-1}}\tfrac {\vert k-l\vert }{\langle l,\xi\rangle^3  }A(l,\xi)A(k-l,\xi-\eta).
\end{align*}
and so 
\begin{align*}
    \tfrac {\vert \eta l- k \xi\vert \vert A(k,\eta )- A(l,\xi) \vert }{\sqrt{(k-l)^2 + (\eta - \xi -(k-l)t )^2 }} \lesssim \langle t \rangle^{-1} \tfrac {1 }{\langle l,\xi\rangle^2  }A(l,\xi)A(k-l,\xi-\eta).
\end{align*}
Therefore,
\begin{align*}
   \vert \calR \vert  &\lesssim \langle t\rangle^{-1}\Vert A a^1\Vert_{L^2}\Vert A p_i \Vert_{L^2}\Vert A a^3\Vert_{L^2}. 
\end{align*}
We integrate in time and use \eqref{eq:Bootvb} to infer 
\begin{align*}
    \int_0^t \vert \calR \vert d\tau &\lesssim c_0 \eps^2.
\end{align*}
\subsection{Average Term }\label{sec:NL=} In this subsection we  bound the nonlinearity of \eqref{eq:NL=neq}, which includes the average in the second component. We split 
\begin{align*}
    NL_{=}&= \langle Aa^1_{\neq} ,A (a^2_=\nabla_t a^3_{\neq})-a^2_= \nabla_t  A a^3_{\neq}\rangle\\
    &=\langle Aa^1_{\neq} ,A (a^2_{1,=}\p_x  a^3_{\neq})-a^2_{1,=}\p_x   A a^3_{\neq}\rangle\\
    &\le  \sum_{k\neq 0 } \iint d(\xi,\eta ) \vert A(k,\eta )- A(k,\xi) \vert \vert k\vert \vert Aa^1\vert (k,\eta ) \vert a^2_1\vert  (0,\eta-\xi) \vert a^3\vert(k,\xi) \\
    &\le \sum_{k\neq 0 } \iint d(\xi,\eta ) (\textbf{1}_{\Omega_{R,=} }+ \textbf{1}_{\Omega_{T,=} } ) \vert A(k,\eta )- A(k,\xi) \vert \vert k\vert \vert Aa^1\vert (k,\eta ) \vert a^2_1\vert  (0,\eta-\xi) \vert a^3\vert(k,\xi)  \\
    &= R_=+T_= 
\end{align*}

according to the sets  
\begin{align*}
    \Omega_{R,=}&=\{\vert k\vert \le 8 \vert \eta-\xi \vert \}, \\
    \Omega_{T,=}&=\{\vert k\vert \ge 8 \vert \eta-\xi \vert \} .
\end{align*}
We estimate $R_=$ and $T_=$ separately. \\
\textbf{Bound on $R_=$:} On $\Omega_{R,=}$ we obtain with Lemma \ref{lem:Jest2}, that $J(k,\eta-\xi ) \le 2 \tilde J(\eta-\xi)$ and so by Lemma \ref{lem:Jest} we infer 
\begin{align}\begin{split}
    \vert A(k,\eta )- A(k,\xi)\vert &\lesssim A(k ,\eta -\xi )\tfrac {J(k,\eta)+J(k,\xi)}{\tilde J(\eta-\xi )} e^ { \lambda(t)(\vert k,\eta \vert^{s }-\vert\eta-\xi \vert^{s })}\\
    &\lesssim A(k,\eta -\xi )\tfrac {J(k,\eta)+J(k,\xi)}{ J(k,\eta-\xi )} e^ { \lambda(t)(\vert k,\eta \vert^{s }-\vert\eta-\xi \vert^{s })}\\
    &\lesssim A(k ,\eta -\xi ) e^{  \lambda(t)\vert l,\xi \vert^{s }}\label{eq:RAdiff2}
\end{split}
\end{align}
and so we obtain 
\begin{align*}
    \vert R_=\vert \lesssim \Vert Aa^1_=\Vert_{L^2}\Vert Aa^2\Vert_{L^2}\Vert Aa^3\Vert_{L^2}.
\end{align*}
We integrate in time and use \eqref{eq:Boot} and \eqref{eq:Bootvb} to infer, that
\begin{align*}
    \int \vert R_= \vert d\tau &\lesssim c_0  \eps^2. 
\end{align*}
For the transport term $T_=$, we split the difference
\begin{align*}
    A(k,\eta )- A(k,\xi)&= A(k,\xi ) (e^{\lambda(t) (\vert k,\eta  \vert^{s }- \vert l,\xi   \vert^{s })}-1 )\\
    &+A(k,\xi ) e^{\lambda(t) (\vert k,\eta  \vert^{s }- \vert k,\xi   \vert^{s })}(\tfrac {J(k,\eta )}{J(k,\xi) }-1 ) \tfrac {\langle k,\eta  \rangle^N}{\langle k,\xi  \rangle^N}\tfrac {m(k,\eta)}{m(k,\xi)}\\
    &+A(k,\xi ) e^{\lambda(t) (\vert k,\eta  \vert^{s }- \vert k,\xi   \vert^{s })}( \tfrac {\langle k,\eta  \rangle^N}{\langle k,\xi  \rangle^N}-1)\tfrac {m(k,\eta)}{m(k,\xi)}\\
    &+A(k,\xi ) e^{\lambda(t) (\vert k,\eta  \vert^{s }- \vert k,\xi   \vert^{s })}( \tfrac {m(k,\eta)}{m(k,\xi)}-1).
\end{align*}
Furthermore on $\Omega_{T,=}$ we obtain 
\begin{align}\begin{split}
    \left \vert e^{\lambda(t) (\vert k,\eta  \vert^{s }- \vert k,\xi   \vert^{s })} \right \vert 
    &\lesssim e^{ \frac 12 \lambda(t)\vert \eta -\xi  \vert^{s }}.\label{eq:T0comm}    
\end{split}
\end{align}
According to this splitting, we obtain 
\begin{align*}
    T_= &=\sum_{k\neq 0 } \iint d(\xi,\eta ) \textbf{1}_{\Omega_{T,=} } \vert k\vert  e^{\lambda \vert k-l,\eta-\xi\vert^s } \vert Aa^1\vert (k,\eta ) \vert a^2_1\vert  (0,\eta-\xi) \vert Aa^3\vert(k,\xi)\\
    &\qquad \cdot \left(  e^{-\frac 12 \lambda (t) \vert k-l,\eta-\xi\vert^s }\vert e^{\lambda(t) (\vert k,\eta  \vert^{s }- \vert l,\xi   \vert^{s })}-1 \vert + \vert \tfrac {J(k,\eta )}{J(k,\xi) }-1 \vert + \vert \tfrac {\langle k,\eta  \rangle^N}{\langle k,\xi  \rangle^N}-1 \vert+ \vert\tfrac {m(k,\eta)}{m(k,\xi)}-1 \vert
    \right)\\
    &=T_{1,=} +T_{2,=} +T_{3,=}+T_{4,=} .
\end{align*}
\textbf{Bound on $T_{1,=}$:} We use $\vert e^x -1\vert \le \vert x\vert  e^{\vert x\vert} $, Lemma \ref{lem:useest} and \eqref{eq:T0comm} to estimate 
\begin{align*}
    \left \vert e^{\lambda(t) (\vert k,\eta  \vert^{s }- \vert k,\xi   \vert^{s })}-1 \right \vert &\lesssim  \left \vert \vert k,\eta  \vert^{s }- \vert k,\xi   \vert^{s }\vert\right\vert  e^{\lambda(t) \left\vert \vert k,\eta  \vert^{s }- \vert k,\xi   \vert^{s }\right\vert}\\
    &\lesssim \tfrac {\vert \eta-\xi \vert } {\vert k ,\eta \vert^{1-s } }e^{ \frac 12 \lambda(t)\vert \eta -\xi  \vert^{s }}. 
\end{align*}
Therefore, with \eqref{eq:Bootvb} we deduce
\begin{align}\begin{split}
    \vert T_{1,=} \vert &\lesssim  \Vert A\Lambda^{\frac s2 } a^1\Vert_{L^2}\Vert  A^{lo}  a^2_=\Vert_{L^2}\Vert A\Lambda^{\frac s2 } a^3\Vert_{L^2}\\
    &\le c_0^{-2} \ln(e+t) \eps^2 \Vert A\Lambda^{\frac s2 } a^1\Vert_{L^2}\Vert A\Lambda^{\frac s2 } a^3\Vert_{L^2}\\
    &\le c_0 \tfrac {\ln(e+t)}{\langle t \rangle^2} \Vert A\Lambda^{\frac s2 } a^1\Vert_{L^2}\Vert A\Lambda^{\frac s2 } a^3\Vert_{L^2}.\label{eq:T=1} 
\end{split}\end{align}
\textbf{Bound on  $T_{2,=}$:} For  $t \le \min(\sqrt { \eta}, \sqrt{ \xi } ) $ we obtain with Lemma \ref{lem:Jcomm}, that 
\begin{align*}
    \vert \tfrac {J(k,\eta ) }{J(k,\xi ) }-1 \vert &\lesssim \tfrac { \langle \eta-\xi \rangle }{\vert \eta , k \vert } e^{ 100\rho\vert k-l,\eta -\xi  \vert^{s }}.
\end{align*}
For $\vert k \vert \ge 4\vert \eta\vert$ we infer with Lemma \ref{lem:Jcomm2}, that 
\begin{align*}
   \vert \tfrac {J(k,\eta ) }{J(k,\xi ) }-1\vert \lesssim \tfrac {1}{\vert k \vert^{\frac 12 } }.
\end{align*}
For $t \ge \min(\sqrt \eta, \sqrt \xi )$ and $ 4\eta \le \vert k\vert  $ we estimate 
\begin{align*}
    \vert k \vert \le \vert k\vert^s t^{2(1-s)}.
\end{align*}
Combining these three estimates we infer 
\begin{align*}
    \vert k \vert \vert \tfrac {J(k,\eta ) }{J(k,\xi ) }-1\vert&\lesssim \vert k\vert^s t^{2(1-s)}e^{ 100\rho\vert \eta -\xi  \vert^{s }}.
\end{align*}
Thus we obtain with \eqref{eq:Boot}, that
\begin{align}\begin{split}
    \vert T_{2,=}\vert &\lesssim t^{2(1-s)}\Vert A\Lambda^{\frac s2 } a^1\Vert_{L^2}\Vert  A^{lo} a^2_=\Vert_{L^2}\Vert A\Lambda^{\frac s2 } a^3\Vert_{L^2}\\
    &\lesssim c_0^{-1} t^{2(1-s)} {\ln(e+t )}\eps^2 \Vert A\Lambda^{\frac s2 } a^1\Vert_{L^2}\Vert A\Lambda^{\frac s2 } a^3\Vert_{L^2}\\
    &\lesssim c_0 t^{-\frac 1 2 -s}  \Vert A\Lambda^{\frac s2 } a^1\Vert_{L^2}\Vert A\Lambda^{\frac s2 } a^3\Vert_{L^2}.\label{eq:T=2}  
\end{split}\end{align}
\textbf{Bound on  $T_{3,=}$:} We use
\begin{align*}
    \left \vert \tfrac {\langle k,\eta  \rangle^N}{\langle k,\xi  \rangle^N}-1\right\vert&\lesssim \tfrac {\langle \eta -\xi  \rangle}{\langle k,\eta  \rangle}
\end{align*}
to infer with \eqref{eq:Boot} that 
\begin{align}\begin{split}
    \vert T_{3,=}\vert&\lesssim \Vert A a^1\Vert_{L^2}\Vert  A^{lo}  a^2_=\Vert_{L^2}\Vert A a^3\Vert_{L^2}\\
     &\lesssim c_0^{-1}\ln(e+t ) \eps^2 \Vert A a^1\Vert_{L^2}\Vert A a^3\Vert_{L^2}.\label{eq:T=3}  
\end{split}\end{align}
\textbf{Bound on  $T_{4,=}$:} With Lemma \ref{lem:mest} we obtain
\begin{align*}
    \vert \tfrac {m(k,\eta)}{m(k,\xi)}-1 \vert &\lesssim \tfrac 1 {\vert k \vert } .
\end{align*}
Thus with \eqref{eq:Boot} we infer, that 
\begin{align}\begin{split}
    \vert T_{4,=}\vert&\lesssim \Vert A a^1\Vert_{L^2}\Vert  A^{lo}  a^2_=\Vert_{L^2}\Vert A a^3\Vert_{L^2}\\
     &\lesssim  c_0^{-1}\ln(e+t ) \eps^2 \Vert A a^1\Vert_{L^2}\Vert A a^3\Vert_{L^2}.\label{eq:T=4}    
\end{split}
\end{align}
\textbf{Conclusion for $T_=$:} Combining the estimates \eqref{eq:T=1}, \eqref{eq:T=2}, \eqref{eq:T=3} and \eqref{eq:T=4} we estimate 
\begin{align*}
    \vert T_=\vert &\lesssim c_0 \tfrac {\ln(e+t)}{\langle t \rangle^2} \Vert A\Lambda^{\frac s2 } a^1\Vert_{L^2}\Vert A\Lambda^{\frac s2 } a^3\Vert_{L^2}\\
    &+c_0 t^{-\frac 1 2 -s}  \Vert A\Lambda^{\frac s2 } a^1\Vert_{L^2}\Vert A\Lambda^{\frac s2 } a^3\Vert_{L^2}\\
    &+c_0 \tfrac {\ln(e+t )}{\langle t\rangle^2} \Vert A a^1\Vert_{L^2}\Vert A a^3\Vert_{L^2}
\end{align*}
and so after integrating in time and using \eqref{eq:Boot}, we infer, that
\begin{align*}
    \int_0^t \vert T_=\vert d\tau&\lesssim (c_0+\tfrac {c_0} \rho )  \eps^2.
\end{align*}
\subsection{Other Nonlinear Terms}\label{sec:ONL} In this subsection we estimate the other nonlinear terms \eqref{eq:Boot2}:
\begin{align*}
        \int_0^t\vert \langle A \p_y^t \Delta^{-1}_t  b  , A (b\nabla_t b- v\nabla_t v) \rangle \vert d\tau&\lesssim c_0 \eps^2,\\
        \int_0^t\vert \langle A \tilde p_1  , A(  \p_y^t \Lambda^{-3}_t \nabla^\perp_t (b\nabla_t v- v\nabla_t b))_{\neq}\rangle\vert d\tau&\lesssim c_0 \eps^2.
\end{align*}
By partial integration we obtain
\begin{align*}
    \vert \langle A \p_y^t \Delta^{-1}_t  b  , A (b\nabla_t b- v\nabla_t v)_{\neq} \rangle\vert &= \vert \langle A \p_y^t \Delta^{-1}_t  \nabla_t \otimes b  , A (b\otimes  b- v\otimes  v)_{\neq} \rangle\vert \lesssim \Vert A (v,b)\Vert_{L^2}^3, \\
    \vert\langle A \tilde p_1  , A  \p_y^t \Lambda^{-3}_t \nabla^\perp_t (b\nabla_t v- v\nabla_t b)_{\neq}\rangle\vert &= \vert\langle A \p_y^t \Lambda^{-3}_t (\nabla^\perp_t \otimes \nabla_t) \tilde p_1  , A(  b\otimes  v- v\otimes  b)_{\neq}\rangle\vert\lesssim \Vert A (v,b)\Vert_{L^2}^3,
\end{align*}
where $\otimes$ is the tensor product. Integrating in time and using \eqref{eq:Boot} yield \eqref{eq:Boot2}. 

\subsection{The Low Frequency Average Estimates}\label{sec:lf} This subsection estimates the low frequent average term \eqref{eq:Boot4}. We rewrite 
\begin{align*}
     \vert NL_=^{lo}\vert  &=\vert \langle A^{lo}  a_{=}^1 , A^{lo} (a^2\nabla_t a^3)_= \rangle \vert   \\
     &\le \sum_{ l\neq 0 } \iint d(\xi ,\eta )   \tfrac {\vert  \eta l \vert   }{\sqrt{l^2 + (\eta - \xi +lt )^2 }} \vert A^{lo}\vert (0,\eta )  \vert A^{lo} a_1\vert (0,\eta )  \vert p_i\vert (-l,\eta-\xi) \vert a^3_1\vert (l,\xi)\\
     &\le \sum_{ k\neq 0 } \iint d(\xi ,\eta )   \tfrac {1}{\sqrt{1 + (\frac{\eta - \xi} k - t )^2 }} \vert \eta A^{lo}(k,\eta )\vert   \vert A^{lo} a \vert (0,\eta )  \vert p_i\vert (k,\eta-\xi) \vert a^3\vert (-k,\xi),
\end{align*}
where we changed $l= -k $. By the definition of $A^{lo}$ \eqref{eq:Adef} and with Lemma \ref{lem:Jest} it holds
\begin{align}
    \vert \eta A^{lo}(\eta)\vert\le \tilde  A(k,\eta ).\label{eq:lowA}
\end{align}
So we obtain 
\begin{align*}
     NL_=^{lo} 
     &\lesssim  \sum_{ k\neq 0 } \iint d(\xi ,\eta )   \tfrac {1}{\sqrt{1 + (\frac{\eta - \xi} k - t )^2 }} \tilde A(k,\eta )  \vert A^{lo} a^1 \vert (0,\eta )  \vert p_i\vert (k,\eta-\xi) \vert a^3\vert (-k,\xi).
\end{align*}
We assume, that we  have the estimate
\begin{align}
\begin{split}
    \tfrac {1}{\sqrt{1 + (\frac{\eta - \xi} k - t )^2 }}&\lesssim \sqrt{\tfrac {\p_t q}q}(\eta-\xi)\left(\sqrt{\tfrac {\p_t q}q}(\eta)+ \tfrac {\vert \eta\vert^{\frac s 2 }}{\langle t \rangle^{s }}\right) \exp(8\rho \vert \xi\vert^{\frac 1 2 }) \textbf{1}_{\vert \eta -\xi\vert \ge  8 \vert \xi \vert} \\
    &+ \tfrac{k^2} {\langle t \rangle }+ \tfrac{\min(\vert \eta-\xi\vert,\vert \xi \vert) } {\langle t \rangle }+\tfrac {\vert \eta-\xi \vert^{\frac s 2 } \vert \eta\vert^{\frac s 2 } }{\langle t \rangle^{2s} }.\label{eq:LF}    
\end{split}
\end{align}
Then, using Lemma \ref{lem:tilA} we infer 
\begin{align*}
    \vert NL_=^{lo} \vert
     &\le \Vert \sqrt{\tfrac {\p_t q}q} A^{lo} a^1_=\Vert_{L^2}\left(\Vert \sqrt{\tfrac {\p_t q}q} \tilde A p_i\Vert_{L^2}+\tfrac 1 {\langle t \rangle^s}   \Vert \Lambda^{\frac s 2 } A p_i\Vert_{L^2}\right)\Vert  A a^1\Vert_{L^2}\\
     &+\tfrac 1{\langle t \rangle } \Vert  A ^{lo} a^1_=\Vert_{L^2}\Vert A p_i\Vert_{L^2}\Vert  A a^1\Vert_{L^2}\\
     &+\tfrac 1{\langle t \rangle^{2s} } \Vert  \Lambda^{\frac s 2 }A ^{lo} a^1_=\Vert_{L^2}\Vert \Lambda^{\frac s 2 } A p_i\Vert_{L^2}\Vert  A a^1\Vert_{L^2}.
\end{align*}
Therefore, after integrating in time, using \eqref{eq:Boot}, \eqref{eq:Bootvb} and $\rho \gtrsim  c_0 $ we obtain
\begin{align*}
    \int_0^t \vert NL_=^{lo}\vert  d\tau&\lesssim  c_0^{-1}(1+\tfrac 1\rho) \ln^2(e+t )\eps^4. 
\end{align*}
Thus \eqref{eq:Boot4} holds if we prove the estimate \eqref{eq:LF}. We distinguish between the cases
\begin{align}
    4\sqrt{\vert \eta-\xi \vert } &\le t \le \tfrac 3 2 \vert \eta -\xi \vert \text{ and }  \vert \eta-\xi \vert \ge 8\min(\vert \xi \vert,\vert k\vert ),\label{eq:c2}\\
    t&\le 4\sqrt{\vert \eta-\xi \vert } \text{ and }  \vert \eta-\xi \vert \ge 8\min(\vert \xi \vert,\vert k\vert ),\label{eq:c3}\\
    \tfrac 3 2 \vert \eta -\xi \vert &\le t \text{ and }  \vert \eta-\xi \vert \ge 8\min(\vert \xi \vert,\vert k\vert ),\label{eq:c4}\\
    \vert \eta-\xi \vert &\le 8\min(\vert \xi \vert,\vert k\vert ) \label{eq:c1}.
\end{align}
If \eqref{eq:c2} holds, we obtain $2\max(\sqrt{\vert \eta-\xi \vert },\sqrt{\vert \xi \vert }) \le t \le 2\min( \vert \eta -\xi \vert,\vert \xi \vert) $ and $\tfrac 12 \vert \eta\vert\le \vert \eta-\xi\vert\le 2\vert \eta\vert$ and therefore by Lemma \ref{lem:qabl} and \ref{lem:qcha} 
\begin{align*}
    \tfrac {1}{\sqrt{1 + (\frac{\eta - \xi} k - t )^2 }}\lesssim \sqrt{\tfrac {\p_t q}q}(\eta-\xi)\left(\sqrt{\tfrac {\p_t q}q}(\eta)+ \tfrac {\vert \eta\vert^{\frac s 2}}{\langle t \rangle^{s}}\right) \exp(8\rho\vert k,\xi\vert^{\frac 1 2 }).
\end{align*}
On the set \eqref{eq:c3}, we obtain 
\begin{align*}
    \tfrac {1}{\sqrt{1 + (\frac{\eta - \xi} k - t )^2 }}\lesssim 1 \le \tfrac {1+\vert \eta-\xi \vert^s}{\langle t \rangle^{2s} }\le 2 \tfrac {\vert \eta-\xi \vert^{\frac s 2 } \vert \eta\vert^{\frac s 2 }}{\langle t \rangle^{2s} }+ \tfrac 2 {\langle t \rangle^{2s}}.
\end{align*}
On the set \eqref{eq:c4}, we use that $\frac{\eta - \xi} k \le \tfrac 2 3 t$ to estimate 
\begin{align*}
    \tfrac {1}{\sqrt{1 + (\frac{\eta - \xi} k - t )^2 }}&\le\tfrac {1}{\sqrt{1 + (\frac t3 )^2 }}\lesssim \tfrac {1}{\langle t \rangle }.
\end{align*}
On the set \eqref{eq:c1}, we obtain 
\begin{align*}
    \tfrac {1}{\sqrt{1 + (\frac{\eta - \xi} k - t )^2 }}\le \tfrac{\langle \eta-\xi \rangle }{\langle t \rangle }&\le  4\tfrac{1+\min ( \vert \eta-\xi \vert, \vert \xi \vert,\vert k \vert ) }{\langle t \rangle }.
\end{align*}
With these estimates, we infer \eqref{eq:LF} and therefore \eqref{eq:Boot4}.

\section{Norm Inflation }\label{sec:LO}
In this section, we prove the lower estimates of Proposition \ref{pro:lest}. In \cite{bedrossian21} lower bounds for the Boussinesq are proven by the Duhamel formula. Here we employ a different approach and prove lower and upper bounds on the linear and nonlinear part using bootstrap. To do this, we split $\tilde p$ of \eqref{eq:tilpI} into linear part $\tilde p_{lin}$ and a nonlinear perturbation $\tilde p-\tilde p_{lin}$ by 
\begin{align}
\begin{split}
    \p_t \tilde p_{lin}&= L \tilde p_{lin}.\\
    \p_t (\tilde p-\tilde p_{lin})&= L(\tilde p-\tilde p_{lin})+ NL [v,b].    \label{eq:tilp2}
\end{split}
\end{align}
with 
\begin{align*}
    L&= \begin{pmatrix}
        0 & \alpha \p_x  - \tfrac 1 {\alpha \p_x } \p_x^4 \Delta_t^{-2}\\
        \alpha \p_x &0
    \end{pmatrix},\\
    NL[v,b]&= \begin{pmatrix}
        \Lambda^{-1}_t  \nabla^\perp_t (b\nabla_t b- v\nabla_t v)_{\neq}- \tfrac 1{\alpha} \p_y^t \Delta^{-1}_t(\Lambda^{-1}_t \nabla^\perp_t (b\nabla_t v- v\nabla_t b))_{\neq}\\
         \Lambda^{-1}_t \nabla^\perp_t (b\nabla_t v- v\nabla_t b)_{\neq}
    \end{pmatrix}. 
\end{align*}
Let $X=\{L^2,H^{-1} \}$. Then for $\tilde K>0$ large enough and initial data that satisfies
\begin{align*}
    \Vert  \tilde p_{in} \Vert_{X} = \tilde K  c_0  \eps,
\end{align*}
we claim that the estimate 
\begin{align}
    C^{-1}_1  c_0 K  \eps\le \Vert \tilde p_{lin}(T)\Vert_{X} \le  C_1 c_0 K  \eps,\label{eq:lowlin}\\
    \Vert \tilde p(T)- \tilde p_{lin}(T)\Vert_{X}\le \tfrac 1 2 C^{-1}_1   c_0 K  \eps,\label{eq:lownl}
\end{align}
 holds for $C_1= e^{\frac \pi {2\alpha }}$ and $\eps T \le c_0 $. From the estimates \eqref{eq:lowlin} and \eqref{eq:lownl} we infer Proposition \ref{pro:lest} directly.

 Let $\tilde m$ be the adapted weight 
\begin{align*}
    \tilde m (t,k,\eta) &= \exp\left( \tfrac 1 {\alpha \vert k\vert }\int_{0 } ^t \tfrac 1 {(1+ (t-\frac \eta k )^2)^2 }\right).
\end{align*}
Then $\tilde m $ satisfies 
\begin{align}
    1\le \tilde m \le C_1.\label{eq:tilm}
\end{align} 
For the linearized term, we obtain 
\begin{align*}
    \p_t \Vert \tilde m \tilde p_{lin} \Vert_{X} &= 2\langle \tilde m \tilde p_{1,lin}, - \tfrac 1 {\alpha \p_x } \p_x^4 \Delta_t^{-2} \tilde m p_{2,lin} \rangle_X + 2\Vert \sqrt {\tfrac {\p_t \tilde m }{\tilde m } } \tilde m p_{lin} \Vert_{X} ^2 \ge 0,\\
    \p_t \Vert \tilde m^{-1}  \tilde p_{lin} \Vert_{X} &= 2\langle\tilde  m^{-1} \tilde p_{1,lin}, - \tfrac 1 {\alpha \p_x } \p_x^4 \Delta_t^{-2}\tilde  m^{-1} p_{2,lin} \rangle_X - 2\Vert \sqrt {\tfrac {\p_t \tilde m }{\tilde m } } \tilde m^{-1} p_{lin} \Vert_{X} ^2 \le 0,
\end{align*}
which yields with \eqref{eq:tilm} the estimate \eqref{eq:lowlin}. 

We prove \eqref{eq:lownl} by bootstrap, we assume that \eqref{eq:lownl} holds for a maximal time $T^\ast <  c_0 \eps^{-1} $ and then improve the estimate. We calculate 
\begin{align*}
    \p_t \Vert \tilde m^{-1}  (\tilde p- \tilde p_{lin}) \Vert_{X}^2 &= 2\langle \tilde m^{-1}  \tilde p_1, - \tfrac 1 {\alpha \p_x } \p_x^4 \Delta_t^{-2} \tilde m^{-1} p_2 \rangle_{X} - 2\Vert \sqrt {\tfrac {\p_t \tilde m }{\tilde m } } \tilde m p \Vert_{X} ^2\\
    &+ 2\langle \tilde m^{-1} ( \tilde p - \tilde p_{lin} ) , \tilde m^{-1}  NL[p]\rangle_{X}. 
\end{align*}
The linear part is bounded by 
\begin{align*}
     2\langle \tilde m^{-1}  \tilde p_1, - \tfrac 1 {\alpha \p_x } \p_x^4 \Delta_t^{-2} \tilde m^{-1} p_2 \rangle_{X} - 2\Vert \sqrt {\tfrac {\p_t \tilde m }{\tilde m } } \tilde m p \Vert_{X} ^2\le 0.
\end{align*}
The nonlinear term we estimate by
    \begin{align*}
        \langle \tilde m^{-1} ( \tilde p - \tilde p_{lin} ) , \tilde m^{-1}  NL[p]\rangle_{X}&\lesssim\Vert \tilde m^{-1} ( \tilde p - \tilde p_{lin} \Vert_{L^2} \Vert ( NL[p])_{\neq} \Vert_{X}\\
        &\lesssim \Vert \tilde m^{-1} ( \tilde p - \tilde p_{lin} )\Vert_{X} \Vert \tilde p \Vert_{H^4 }^2. 
    \end{align*}
Using \eqref{eq:Boot} and \eqref{eq:lownl} we infer 
\begin{align*}
    \p_t \Vert \tilde m^{-1}  (\tilde p- \tilde p_{lin}) \Vert_{X}^2&\lesssim  c_0 \tilde K  \eps^3
\end{align*}
and by integrating in time and using $\eps T^\ast < c_0 $ we use  
\begin{align*}
    \Vert \tilde m^{-1}  (\tilde p- \tilde p_{lin}) \Vert_{X}^2&\lesssim  c_0^2\tilde K  \eps^2= \tfrac 1{\tilde K}  ( c_0 \tilde K\eps)^2.
\end{align*}
Thus for $\tilde K$ large enough we obtain 
\begin{align*}
    \Vert (\tilde p- \tilde p_{lin}) \Vert_{L^\infty _TX}^2&<  \tfrac 1 2 C_1^{-1} (c_0 \tilde K  \eps)^2. 
\end{align*}
This contradicts the maximality of $T^\ast$ and so \eqref{eq:lownl} holds for $\eps T \le c_0 $.
\qed

\section{Dissipative Estimates}\label{sec:dissi}
In this section, we prove Corollary \ref{cor:dissi}. The proof of Corollary \ref{cor:dissi} is mostly similar to the proof of Theorem \ref{Thm:idealmain} with some, possible anisotropic dissipation. We only stretch the differences given by dissipation. We consider the dissipative MHD equation 
\begin{align}
  \label{MHD2dissi}
  \begin{split}
    \partial_t V+y\p_x V + V_2e_1 &+ V\cdot \nabla V- 2\p_x \nabla \Delta^{-1} V_2 + \nabla \Pi  =\nu \Delta V+ B\cdot\nabla B, \\
    \partial_t B+y\p_x B- B_2e_1 &+ V\cdot\nabla B\qquad \qquad \qquad \qquad \quad \, = \kappa \Delta B+B\cdot\nabla V, \\
    \nabla\cdot V=\nabla\cdot B  &= 0,
  \end{split}
\end{align}
When we change to the adapted unknowns \eqref{eq:tilp} we obtain 
\begin{align*}
\begin{split}
    \p_t \tilde p_1
    &= I_1 [\tilde p] + \nu\Delta_t \tilde p_1-(\kappa-\nu)\p_y^t \Delta_t^{-1}\tilde  p_2 ,\\
    \p_t \tilde p_2     &= I_2 [\tilde p] + \kappa \Delta_t \tilde p_2 .    
\end{split}
\end{align*}
where $I$ are the linear and nonlinear terms of the ideal setting. Then we obtain the estimate 
\begin{align*}
    \p_t \Vert A(\tilde v,\tilde b) \Vert_{L^2}^2&=\calI +\langle Ap_1 , (\kappa-\nu) \p_y^t \tilde p_2 \rangle - \nu \Vert \tilde p_1 \Vert_{L^2 } - \kappa \Vert \tilde p_2 \Vert_{L^2 }
\end{align*}
where 
$$\calI=-2 \vert \dot \lambda\vert \Vert \tilde A\Lambda^{\frac s2} (\tilde p,v_=,b_=)\Vert_{L^2}^2-  2 \Vert \sqrt{\tfrac {\p_t q} q} A(\tilde p,v_=,b_=) \Vert_{L^2}^2+L+ 2NL+2 ONL $$
are the terms of the ideal setting. The $\calI$ can be treated exactly as in the ideal case. Only the dissipation terms are different. We estimate 
\begin{align*}
    \langle Ap_1 , (\kappa-\nu) \p_y^t \tilde p_2 \rangle =&\le \nu \Vert \p_y^t \tilde p_1 \Vert_{L^2 }\Vert \tilde p_2 \Vert_{L^2 } +\kappa \Vert \p_y^t \tilde p_2 \Vert_{L^2 }\Vert \tilde p_1 \Vert_{L^2 }\\
    &\le (\nu+ \kappa)\Vert A \tilde p \Vert_{L^2 }^2 +\nu \Vert \tilde p_1 \Vert_{L^2 } + \kappa \Vert \tilde p_2 \Vert_{L^2 }
\end{align*}
Integrating in time yields 
\begin{align*}
    \int_0^t \langle Ap_1 , (\kappa-\nu) \p_y^t \tilde p_2 \rangle - \nu \Vert \tilde p_1 \Vert_{L^2 } - \kappa \Vert \tilde p_2 \Vert_{L^2 }d\tau \le (\nu+\kappa)t \Vert \tilde p \Vert_{L^2 }^2 .
\end{align*}
Therefore, we can bound the dissipation error by dissipation if $\max (\nu,\kappa)t\le c_0 $.

\subsection*{Data availability}
No data was used for the research described in the article.

\subsection*{Acknowledgements}
The author wants to thank Christian Zillinger and Michele Dolce for their useful comments.

N. Knobel has been funded by the Deutsche Forschungsgemeinschaft (DFG, German Research Foundation) – Project-ID 258734477 – SFB 1173 while working at the Karlsruhe Institute of Technology.\\
N. Knobel is supported by ERC/EPSRC Horizon Europe Guarantee EP/X020886/1 while working at Imperial College London.

The author declares that he has no conflict of interest.
\appendix

\section{Properties of the Weights }\label{sec:q}

In this section, we construct the weight $q$ and prove properties of $q$, $J$ and $m$. We remark that this $q$ is an adaption of the weight $w_{NR}$ defined and used in \cite{bedrossian2013inviscid,bedrossian21} with a different constant $\rho$.

\subsection{Construction of $q$}
Our main growth model (see Subsection \ref{sec:NLgrow}) corresponds to 
\begin{align*}
    \p_t q &\approx \rho\tfrac 1{1+\vert t-\frac \eta k  \vert }q
\end{align*}
for $\rho>0$ and $t\in I_{k,\eta }$ (In the proof of Theorem \ref{Thm:ideal} we assume that $\tfrac \rho c_0$ is large enough independent of $c_0$). Let $\eta >0$ and $1\le  k \le \lfloor\sqrt{  \eta  }\rfloor$, then for 
\begin{align*}
    t_{ k ,\eta}= \tfrac 1 2 (\tfrac \eta k + \tfrac \eta{k+1} ),\qquad t_{0,\eta} = 2\vert \eta\vert
\end{align*}
we define the weight $q$ as 
\begin{align*}
    q(t,\eta ) &= \left(\tfrac {k^2} \eta ( 1+ b_{k,\eta } \vert t-\tfrac \eta k \vert \right)^{\rho } q(t_{k,\eta},\eta )  &t&\in [\tfrac \eta k , t_{ k -1,\eta } ]\\
    q(t,\eta ) &= (1+ a_{k,\eta } \vert t-\tfrac \xi k \vert)^{\rho } q(\tfrac \eta k ,\eta )   &t&\in [t_{k,\eta },\tfrac \eta k  ]
\end{align*}
with 
\begin{align*}
    b_{k,\eta} &= \tfrac {2(k-1)}k (1-\tfrac {k^2} \eta ),\\
    a_{k,\eta} &= \tfrac {2(k+1)}k (1-\tfrac {k^2} \eta ).
\end{align*}
The $a_{k,\eta}$ and $b_{k,\eta}$ are choosen to satisfy $ 1+ b_{k,\eta } \vert t_{k-1,\eta}-\tfrac \eta k \vert=1+ a_{k,\eta } \vert t_{k,\eta}-\tfrac \eta k \vert=\tfrac \eta {k^2}$. For $\eta<0$, we define $q(\eta)= q(-\eta)$.

For $\eta >0$, the weight $q$ corresponds to $w_{NR}$ of \cite{bedrossian2013inviscid} with a different constant in the exponent. In the following, we list of properties of $q$ done in \cite{bedrossian2013inviscid}. 

\begin{lemma}[\cite{bedrossian2013inviscid}, Lemma 3.1]\label{lem:qest} Let $\vert  \eta\vert  >1$, then we obtain that 
    \begin{align*}
        \tfrac 1 {q(\eta )}\sim  \tfrac 1 {\eta^{\rho  }} \exp(8\rho   \vert \eta\vert^{\frac 1 2 } ).%
    \end{align*}
\end{lemma}
\begin{proof}
The proof is the same as in \cite{bedrossian2013inviscid}.
\end{proof}

\begin{lemma}[\cite{bedrossian2013inviscid}, Lemma 3.3]\label{lem:qabl}
    Let $t\in I_{\eta,k}$ and $2\sqrt {\vert \eta\vert}  \le t\le 2\eta$ we obtain 
        \begin{align*}
            \sqrt{\tfrac{\p_tq }q}(\eta ) &\approx\rho  \tfrac 1 {1+\vert t-\frac \eta k\vert  } .
        \end{align*}
\end{lemma}
\begin{proof}
    Follows from the definition of $q$. 
\end{proof}

\begin{lemma}[\cite{bedrossian2013inviscid}, Lemma 3.4]\label{lem:qcha}
    Let  $t\le 2\min(\eta,\xi) $ and $\tfrac 1 2 \vert \xi\vert \le \vert \eta \vert \le 2\vert \xi\vert $
        \begin{align*}
            \sqrt{\tfrac{\p_tq }q}(\xi) &\lesssim \left(\sqrt{\tfrac{\p_tq }q}(\eta ) + \tfrac {\vert \eta\vert^{\frac s 2 }}{\langle t \rangle^s }\right)\langle \eta-\xi\rangle .
        \end{align*}
\end{lemma}

\begin{proof}
    The proof is the same as in \cite{bedrossian2013inviscid}, with the choice of $k$ and $l$ such that $t\not\in I_{l,\xi}\cup I_{k,\eta }$.  
\end{proof}

\begin{lemma}[\cite{bedrossian2013inviscid}, Lemma 3.5]\label{lem:qdiff}
    For all $t ,\eta,\xi $ we obtain 
        \begin{align*}
            \tfrac {q(t,\xi )}{q(t,\eta )}\lesssim e^{8\rho \vert \eta-\xi \vert^{\frac 1 2 } }.
        \end{align*}
\end{lemma}
\begin{proof}
    The proof is the same as in \cite{bedrossian2013inviscid}. 
\end{proof}

\subsection{Properties of $J $}
We define $J$ (c.f.  \cite{bedrossian2013inviscid} ) as
\begin{align*}
    J(k,\eta) &= \tfrac {e^{8 \rho \vert \eta \vert^{\frac 1 2 }}}{q( \eta)}+ e^{8 \rho \vert k\vert^{\frac 1 2 }}. 
\end{align*}

\begin{lemma} \label{lem:Jest}
    For all $\eta,\xi,k,l $, the weight $J$ satisfies
    \begin{align*}
        1\le J(k,\eta )\le 2 \exp( 8\rho \vert k,\eta \vert^{\frac 1 2 } ). 
    \end{align*}
    and 
    \begin{align*}
        \tfrac { J(k, \eta )}{J(l,\xi)}&\le 2 \exp( 8\rho\vert k-l,\eta-\xi \vert^{\frac 1 2 } ),\\
        \tfrac {\tilde J(\eta )}{\tilde J(\xi)}&\le 2 \exp( 8\rho\vert \eta-\xi \vert^{\frac 1 2 } ). 
    \end{align*}
\end{lemma}

\begin{lemma}\label{lem:Jest2}
    For all $\eta,\xi,k,l $ such that $4\vert k \vert \le \vert \eta \vert $, the weight $J$ satisfies
    \begin{align*}
        J(k,\eta) \le 2 \tilde J (k,\eta ) . 
    \end{align*}
\end{lemma} 
\begin{proof}
    Follows from Lemma \ref{lem:qest}.
\end{proof}

\begin{lemma}[\cite{bedrossian2013inviscid}, Lemma 3.7]\label{lem:Jcomm}
     For all $\eta,\xi,k,l $ such that $t \le\tfrac 1 2 \min(\sqrt{\vert \eta \vert }, \sqrt{\vert \xi  \vert }) $, the weight $J$ satisfies
        \begin{align*}
            \vert \tfrac {J(t,k,\eta )} {J(t,l,\xi  )}-1\vert \lesssim \tfrac{\langle \eta-\xi,k-l  \rangle } {\sqrt{\vert \eta\vert +\vert \xi \vert+\vert k\vert +\vert l \vert }} e^{100\rho \vert \eta-\xi \vert^{\frac 1 2 } }.
        \end{align*}
\end{lemma}

\begin{proof}
 The proof is the same as in \cite{bedrossian2013inviscid}.
\end{proof}

\begin{lemma}\label{lem:Jcomm2}
     For all $\eta,\xi,k,l $ such that $4\eta \le \vert l\vert $, the weight $J$ satisfies
        \begin{align*}
    \vert \tfrac {J(k,\eta ) }{J(l,\xi ) }-1\vert &\lesssim \tfrac 1 \rho \tfrac {\langle k-l\rangle }{\vert k \vert^{\frac 12 } }\exp(8\rho  (\vert k-l\vert^{\frac 1 2 }).
\end{align*}
\end{lemma}

\begin{proof}

    For $4\eta \le \vert l\vert$ we use that 
\begin{align*}
    \left \vert \tfrac {J(k,\eta ) }{J(l,\xi ) }-1 \right\vert &= \left \vert\tfrac {\tfrac {e^{8\rho \vert \eta \vert^{\frac 1 2 }}}{q(\eta )} + e^{8\rho\vert k  \vert^{\frac 1 2 }} }{\tfrac {e^{8\rho\vert \xi \vert^{\frac 1 2 }}}{q(\xi )} + e^{8\rho \vert l  \vert^{\frac 1 2 }} }-1\right\vert \\
    &\le \left \vert\exp(8\rho (\vert k\vert^{\frac 1 2 }- \vert l \vert^{\frac 1 2 }) ) -1\right \vert + \exp(8\rho(\vert \eta\vert^{\frac 1 2 }- \vert l \vert^{\frac 1 2 }) )\\
    &\le \tfrac {\vert k-l\vert}{\vert k \vert^{\frac 12 } }\exp(8\rho (\vert k-l\vert^{\frac 1 2 })) + \exp(- 4\rho  \vert l \vert^{\frac 1 2 } ) \\
    &\lesssim \tfrac 1 \rho \tfrac {\langle k-l\rangle }{\vert k \vert^{\frac 12 } }\exp(8\rho(\vert k-l\vert^{\frac 1 2 })).
\end{align*}
\end{proof}

\begin{lemma}\label{lem:tilA}
    Let $\tfrac 1 2 \le s \le \tfrac 3 4 $, then we obtain for $\tilde A(0,\eta)$, that 
    \begin{align*}
        \tilde A(0,\eta )&\lesssim \tilde A(k,\xi )\tilde A(k,\eta-\xi) (\tfrac 1{\langle k ,\eta\rangle^N}+\tfrac 1{\langle k ,\xi\rangle^N}),\\
        \textbf{1}_{\vert \eta-\xi\vert\ge 8\vert \xi\vert} \tilde A(0,\eta )&\lesssim \tilde A(k,\xi )\tilde A(k,\eta-\xi) e^{-\frac 1 2 \lambda(t) \vert \xi \vert^{\frac 1 2 } }\tfrac 1{\langle k ,\xi\rangle^N}.
    \end{align*}

    \begin{proof}
        WLOG we assume $\eta -\xi \ge \xi \ge 0 $. By Lemma \ref{lem:Jest} we obtain 
        \begin{align*}
        \tfrac {\tilde J(\eta )}{\tilde J(\xi)\tilde J(\eta-\xi)}&\le 2 \exp( 8\rho\min(\vert \eta-\xi \vert^{\frac 1 2 },\vert \xi \vert^{\frac 1 2 }) ). 
    \end{align*}
     With Lemma \ref{lem:useest} and $\lambda(t) \ge 250 \rho $ we obtain
     \begin{align*}
        \textbf{1}_{\vert \eta-\xi \vert \le 8\vert \xi\vert } e^{\lambda(t)\vert \eta \vert^{s }}&\le  \textbf{1}_{\vert\eta- \xi \vert \le 8\vert \xi\vert } e^{\lambda(t)(\tfrac 89)^{1-s} (\vert \xi \vert^{s }+\vert \eta- \xi \vert^{s })}\\
         &\le \textbf{1}_{\vert \eta-\xi \vert \le 8\vert \xi\vert } e^{(\lambda(t)-8 \rho )(\vert \xi \vert^{s }+\vert \eta- \xi \vert^{s })}
     \end{align*}
     and 
     \begin{align*}
        \textbf{1}_{8\vert \xi \vert \le \vert \eta-\xi\vert } e^{\lambda(t)\vert \eta \vert^{s }}&\le  \textbf{1}_{8\vert \xi \vert \le \vert \eta-\xi\vert } e^{\lambda(t)(\vert\eta-\xi \vert^{s }+ \frac 1 2 \vert\xi \vert^{s }) }.
     \end{align*}
        Therefore, the claim holds by the definition of $A$. 
    \end{proof}
\end{lemma}

\subsection{Properties of $m$}

We consider the time-dependent Fourier weight 
\begin{align*}
    m(k,\eta ) &= 
    \begin{cases}
    \exp(- \tfrac 1 {\vert k\vert } \int_{-\infty }^t \tfrac 1 {1+(\frac \eta k -\tau )^2  }d\tau)     & k\neq 0  \text{ and } \sqrt \eta \le 10 c_0 \eps^{-1} \\
    1&\text{else}.
    \end{cases}
\end{align*}
By straightforward calculation, $m$ satisfies the following two lemmas:
\begin{lemma}\label{lem:mapprox}
    The weight $m$ satisfies for all $k,\eta$
    \begin{align*}
       1\le m(k,\eta ) \le e^{\pi} .
    \end{align*}
\end{lemma}

\begin{lemma}\label{lem:mest}
    The weight $m$ satisfies for all $l,k,\xi,\eta$
    \begin{align*}
       \vert {m(k,\eta)}-{ m(l,\xi)} \vert  &\lesssim 
       \begin{cases}
       \tfrac {\vert k-l\vert} {\min(\vert k \vert, \vert l \vert) }\textbf{1}_{\min(\sqrt \eta, \sqrt \xi)\le \sqrt c_0 \eps^{-1} } & k,l \neq 0\\
       \tfrac 1 {\vert k \vert } \textbf{1}_{\min(\sqrt \eta, \sqrt \xi)\le\sqrt  c_0 \eps^{-1} } & k \neq 0, l =0 \\
       0 &k=l=0.
       \end{cases}
    \end{align*}
\end{lemma}

\section{Gevrey Related Estimates}

In this section, we summarize a few inequalities related to Gevrey spaces. 
\begin{lemma}[\cite{bedrossian2013inviscid}, Lemma A.2]\label{lem:useest}
    Let $0<s<1$ and $x\ge y\ge 0$
    \begin{itemize}
        \item For $x\neq 0$, it holds
        \begin{align}
            \vert x^s-y^s \vert &\lesssim \tfrac 1 {x^{1-s}+y^{1-s}}\vert x-y\vert.
        \end{align}
         \item For $\vert x-y \vert\le \tfrac x K $ and $K>1$, it holds
        \begin{align}
            \vert x^s-y^s \vert &\le \tfrac s {(K-1)^{1-s }}\vert x-y\vert^s   
        \end{align}
        \item It holds that
        \begin{align}
            \vert x+y \vert^s &\le \left(\tfrac x{x+y} \right)^{1-s }(x^s+y^s).
        \end{align}
        In particular, for $y\le x\le K y $ for some $K$, 
        \begin{align}
            \vert x+y\vert^s \le \left(\tfrac K{1+K} \right)^{1-s }(x^s+y^s) 
        \end{align}
    \end{itemize}
\end{lemma}
By straightforward estimates and using Sobolev embedding, we infer the following product lemma: 
\begin{lemma}[Algebra property]
    Let $N>2$ and $0 <s\le 1$, then we obtain for $f,g\in \calG^{r } $
    \begin{align*}
        \Vert f g \Vert_{\calG^{r }}\lesssim \Vert f\Vert_{\calG^{r }}\Vert g\Vert_{\calG^{r }}
    \end{align*}
    independent of $r$.  
\end{lemma}

\bibliographystyle{alpha} 

\bibliography{library}

\end{document}